\documentclass[11pt,reqno]{amsart}
\newcommand{\mysection}[1]{\section{#1}
      \setcounter{equation}{0}}

\newcommand\cbrk{\text{$]$\kern-.15em$]$}} 
\newcommand\opar{\text{\raise.2ex\hbox{${\scriptstyle | }$}\kern-.34em$($} }

\usepackage{color}
\usepackage{amsmath,amsthm,amssymb,amsfonts,enumerate,color,enumerate}
\usepackage[pdftex]{graphicx}

\usepackage{verbatim}
\usepackage[spanish,USenglish]{babel} 
\usepackage{color}
\usepackage{tikz}
\allowdisplaybreaks

\DeclareMathOperator*{\esssup}{ess\,sup}

\newtheorem{theorem}{Theorem}[section]
\newtheorem{lemma}[theorem]{Lemma}
\newtheorem{proposition}[theorem]{Proposition}
\newtheorem{corollary}[theorem]{Corollary}

\theoremstyle{definition}
\newtheorem{assumption}{Assumption}[section]
\newtheorem{definition}{Definition}[section]

\theoremstyle{remark}
\newtheorem{remark}{Remark}[section]

\newcommand{\R}{\mathbb{{R}}}

\newcommand{\N}{\mathbb{{N}}}

\newcommand\bH{\mathbb{H}}
\newcommand\bI{\mathbb{I}}

\newcommand\bL{\mathbb{L}}
\newcommand\bR{\mathbb{R}}

\newcommand\bS{\mathbb{S}}

\newcommand\bU{\mathbb{U}}
\newcommand\bV{\mathbb{V}}
\newcommand\bW{\mathbb{W}}

\newcommand\cA{\mathcal{A}}
\newcommand\cB{\mathcal{B}}

\newcommand\cF{\mathcal{F}}

\newcommand\cH{\mathcal{H}}

\newcommand\cJ{\mathcal{J}}
\newcommand\cK{\mathcal{K}}
\newcommand\cL{\mathcal{L}}
\newcommand\cM{\mathcal{M}}
\newcommand\cN{\mathcal{N}}
\newcommand\cP{\mathcal{P}}

\newcommand\cZ{\mathcal{Z}}

\makeatletter
 \newcommand{\sumstar}
 {\operatornamewithlimits{\sum@\kern-.2em\raise1ex\hbox{*}}}
 \makeatother

\newcommand{\<}{\langle}
\renewcommand{\>}{\rangle}

\begin{document}

\author[I. Gy\"ongy]{Istv\'an Gy\"ongy}
\address{School of Mathematics and Maxwell Institute, University of Edinburgh, Scotland, United Kingdom.}
\email{i.gyongy@ed.ac.uk}

\author[S. Wu]{Sizhou Wu}
\address{School of Mathematics,
University of Edinburgh,
King's  Buildings,
Edinburgh, EH9 3JZ, United Kingdom}
\email{Sizhou.Wu@ed.ac.uk}

\keywords{Stochastic integro-differential equation, Bessel potential spaces, Interpolation couples}

\subjclass[2010]{Primary 45K05, 60H15, 60H20, 60J75; Secondary 35B65}

\begin{abstract} A class of (possibly) degenerate stochastic 
integro-differential equations of parabolic type is considered, 
which includes the Zakai equation in nonlinear filtering for jump diffusions. 
Existence and uniqueness of the solutions are established in Bessel potential  
spaces. 
\end{abstract}

\title[Integro-Differential Equations]{On $L_p$-Solvability 
of Stochastic Integro-Differential Equations}

\maketitle

\mysection{Introduction}
We consider the  equation 
\begin{align}                                              
du_t(x)= &(\cA_tu_t(x)+f_t(x))\,dt
+\left( \cM^r_tu_t(x)+g^r_t(x) \right)\,dw^r_t\nonumber\\
&+\int_Z\left( u_{t-}(x+\eta_{t,z}(x))-u_{t-}(x)+h_t(x,z) \right)\,\tilde{\pi}(dz,dt)  \label{eq1}
\end{align}
for $(t,x)\in [0,T]\times \bR^d:=H_T$ with initial condition 
\begin{equation}                                                                                                 \label{ini1}
u_0(x)=\psi(x) \quad \text{for $x\in \bR^d$},  
\end{equation}
on a filtered probability space $(\Omega,\cF,P,(\cF_t)_{t\geq0})$,  
carrying a sequence $w=(w^i_t)_{i=1}^{\infty}$ of independent 
$\cF_t$-Wiener processes and an $\cF_t$-Poisson martingale measure 
$\tilde\pi(dz,dt)=\pi(dz,dt)-\mu(dz)\otimes dt$, where $\pi(dz,dt)$ is an 
$\cF_t$-Poisson random measure with a $\sigma$-finite intensity measure $\mu(dz)$ 
on a measurable space $(Z,\cZ)$ with countably generated $\sigma$-algebra $\cZ$. 
We note that here, and later on, the summation convention is used with respect to 
repeated (integer-valued) indices and multi-numbers. 

In the above equation $\cA_t$ is an integro-differential operator 
of the form $\cA_t=\cL_t+\cN_t^{\xi}+\cN^{\eta}_t$ with 
a second order differential operator 
$$
\cL_t=a^{ij}_t(x)D_{ij}+b^i_t(x)D_i+c_t(x),  
$$
and integral operators $\cN^{\xi}_t$ and $\cN^{\eta}_t$ defined by 
$$                                                     
\cN^{\xi}_t\varphi(x)=\int_Z \varphi(x+\xi_{t,z}(x))-\varphi(x)
-\xi_{t,z}(x)\nabla\varphi(x)\,\nu(dz)
$$
\begin{equation}                                                               \label{def N}
\cN^{\eta}_t\varphi(x)=\int_Z \varphi(x+\eta_{t,z}(x))-\varphi(x)
-\eta_{t,z}(x)\nabla\varphi(x)\,\mu(dz)
\end{equation}
for a suitable class of real-valued functions $\varphi$ on $\bR^d$ for each $t\in[0,T]$, 
where $\nu(dz)$ is a fixed $\sigma$-finite measure on $(Z,\cZ)$. 
For each integer $r\geq1$ the operator $\cM_t^r$ is a first order differential operator 
of the form 
\begin{equation*}                                                           
\cM^r_t=\sigma^{ir}_t(x)D_i+\beta_t^r(x).
\end{equation*}
The coefficients $a^{ij}$, $b^i$, $c$, $\sigma^{ir}$ and $\beta^{r}$ are real functions 
on $\Omega\times H_T$ for $i,j=1,2,...,d$ and integers $r\geq1$,  and 
$\eta=(\eta^i_{t,z}(x))$ and $\xi=(\xi^i_{t,z}(x))$ are $\bR^d$-valued 
functions  
of $(\omega,t,x,z)\in\Omega\times H_T\times Z$. The free terms $f$ and $g^r$ 
are real functions defined 
on $\Omega\times H_T$ for every $r\geq1$, and $h$ is a real function defined 
on $\Omega\times H_T\times Z$.  The stochastic differentials in equation 
\eqref{eq1} are understood in It\^o's sense, see the definition of a solution in the next 
section. 

We are interested in the solvability of the above problem in $L_p$-spaces. 
We note that equation \eqref{eq1} may degenerate, i.e., the pair of linear operators 
$(\cL,\cM)$ satisfies only the {\it stochastic parabolicity} condition, 
Assumption \ref{assumption L} below, and the operator $\cN^{\xi}$ 
may also degenerate. Our main result, Theorem \ref{theorem main} states that under 
the stochastic parabolicity condition  on the operators $(\cL,\cM)$, 
$\cN^{\xi}$, $\cN^{\eta}$, and appropriate 
regularity conditions on their coefficients and on the initial and free data,  
the Cauchy problem \eqref{eq1}-\eqref{ini1} has a unique generalised solution 
$u=(u_t)_{t\in[0,T]}$ for a given $T$. Moreover, this theorem 
describes the temporal and spatial regularity of $u$ in terms 
of Bessel potential spaces $H^n_p$, and presents also a supremum estimate in time.   
The uniqueness of the solution 
is proved by an application of a theorem on It\^o's 
formula from \cite{GW2019}, which generalises a theorem of Krylov 
in \cite{K2010} to the case of jump processes. 
The existence of a generalised 
solution is proved in several steps. First we obtain 
a priori estimates in Sobolev spaces $W^n_p$ 
for integers $n\in[0,m]$ if $p=2^k$ for an integer $k\geq1$, 
where $m$ is a parameter 
measuring the spatial smoothness of the coefficients and the data 
in \eqref{eq1}-\eqref{ini1}. These estimates allow us to 
construct a generalised solution by standard methods 
of approximating 
\eqref{eq1}-\eqref{ini1} with non-degenerate equations 
with smooth coefficients and compactly 
supported smooth data in $x\in\bR^d$, and passing 
to the weak limit in appropriate spaces. 
Thus we see that a solution operator, mapping 
the initial and free data into the solution 
of \eqref{eq1}-\eqref{ini1}, exists and it is a bounded linear operator 
in appropriate  
$L_p$-spaces if $p=2^k$ for an integer $k\geq1$. 
Hence by interpolation we get our 
a priori estimates in Bessel potential spaces $H^n_p$ 
for any given $p\geq 2$ and 
real number $n\in[0,m]$. 
We obtain essential supremum estimates in time for the solution 
from integral estimates, by using the simple fact 
that the essential supremum of Lebesgue functions 
over an interval $[0,T]$ is the limit of their 
$L_r([0,T])$-norm as $r\to\infty$. 
Hence we get the temporal regularity of 
the solution formulated in our main theorem by using 
Theorem 2.2 on It\^o's formula  in \cite{GW2019}, 
an extension of Lemma 5.3 in \cite{DGW} 
and a well-known interpolation inequality,  
Theorem \ref{theorem i1}(v) below. 

Concerning the above construction of a generalised 
solution in $L_p$-spaces we would like 
to emphasise that first we can get the necessary a priori 
estimates only if $p=2^k$ for an integer $k\geq1$ 
and we need to use interpolation via the solution 
operator to get these estimates 
for arbitrary $p\geq2$. We note that a similar situation arose in 
$L_p$-estimates in finite difference approximations for 
stochastic PDEs in \cite{GG}.

In the literature there are many results 
on stochastic integral equations  
with unbounded operators, driven by 
jump processes and martingale measures. 
A general existence and uniqueness theorem 
for stochastic evolution equations 
with nonlinear
operators satisfying stochastic coercivity 
and monotonicity conditions is proved in 
\cite{G1982}, which generalises some 
results in \cite{Pa} and \cite{KR1979} to  
stochastic evolution equations driven 
by semimartingales and random measures. 
This theorem implies the existence of a unique generalised solution 
to \eqref{eq1}-\eqref{ini1} in $L_2$-spaces when instead of 
the stochastic parabolicity condition 
\eqref{parabolicity} in Assumption \ref{assumption L} below, the 
strong stochastic parabolicity condition, 
$$
\sum_{i,j=1}^d\alpha^{ij}z^iz^j\geq\lambda \sum_{i=1}^d|z^i|^2
\quad\text{for all $z=(z^1,z^2,...,z^d)\in\bR^d$}
$$
with a constant $\lambda>0$ is assumed on $(\cL,\cM)$.  
Under the weaker condition of stochastic parabolicity 
the solvability of \eqref{eq1}-\eqref{ini1} 
in $L_2$-spaces is investigated and existence 
and uniqueness theorems are presented 
in \cite{Da} and \cite{LM2015}. The first result  
on solvability in $L_p$-spaces for the stochastic 
PDE problem \eqref{eq1}-\eqref{ini1} with 
$\xi=\eta=0$ and $h=0$ was obtained in \cite{KR}, 
and was improved in \cite{GK2003}. However, there is a gap 
in the proof of the crucial a priori estimate in 
\cite{KR}. This gap is filled in 
and more general results on solvability in 
$L_p$-spaces for systems of stochastic 
PDEs driven by Wiener processes are presented 
in \cite{GGK}. As far as we know Theorem \ref{theorem main} 
below is the first result 
on solvability in $L_p$-spaces of stochastic integro-differential 
equations (SIDEs) without 
any non-degeneracy conditions. It generalises 
the main result of \cite{DGW} on deterministic 
integro-differential equations to SIDEs. 
Our motivation to study equation \eqref{eq1} comes from nonlinear filtering of 
jump-diffusion processes, and we want to apply 
Theorem \ref{theorem main} to filtering problems 
in a continuation of the present paper. 

We note that under non-degeneracy conditions SIDEs 
have been investigated with various generalities 
in the literature, and very nice results on their solvability 
in $L_p$-spaces have recently been obtained. In particular,  
$L_p$-theories for 
such equations have been developed in 
\cite{KK}, \cite{KL}, \cite{MPh1}, 
\cite{MPh3} and \cite{MP4}, which extend some results 
of the $L_p$ theory of Krylov 
\cite{K1996} to certain classes of equations with non local operators. 
See also \cite{CL}, \cite{DK} and \cite{Z} 
in the case of deterministic equations.  
Nonlinear filtering problems and the related equations 
describing the conditional distributions 
have been extensively studied in the literature. 
For results in the case of jump-diffusion 
models see, for example, 
\cite{AB}, \cite{B}, \cite{FH} and \cite{Gr1976}.

In conclusion, we introduce some notions and notations used throughout this paper.
All random elements are given on the filtered probability space 
$(\Omega,\cF,P,(\cF_t)_{t\geq0})$. We assume that $\cF$ is $P$-complete, 
the filtration $(\cF_t)_{t\geq0}$ is right-continuous, and $\cF_0$ contains 
all $P$-zero sets of $\cF$. The $\sigma$-algebra  
of the predictable subsets of $\Omega\times[0,\infty)$ is denoted by $\cP$. 
For notations, notions and results concerning L\'evy processes, Poisson 
random measures and stochastic integrals we refer to \cite{A2009}, 
\cite{BMR} and \cite{IW2011}.

For vectors $v=(v^i)$ and $w=(w^i)$ in $\bR^d$ we 
use the notation $vw=\sum_{i=1}^mv^iw^i$ and $|v|^2=\sum_i|v^i|^2$. 
For real-valued Lebesgue measurable functions 
$f$ and $g$ defined on $\bR^d$ the notation $(f,g)$ 
means the integral of the product $fg$ over $\bR^d$ 
 with respect to the Lebesgue measure on $\bR^d$.   
A finite list $\alpha=\alpha_1\alpha_2,...,\alpha_n$  of numbers $\alpha_i\in\{1,2,...,d\}$ 
is called a multi-number of length $|\alpha|:=n$, and the notation 
$$
D_{\alpha}:=D_{\alpha_1}D_{\alpha_2}...D_{\alpha_n}
$$ 
is used for integers $n\geq1$, 
where 
$$
D_i=\frac{\partial}{\partial x^i}, \quad \text{for $i\in\{1,2,...,d\}$}. 
$$
We use also the multi-number $\epsilon$ of length $0$ 
such that $D_{\epsilon}$ means the identity operator. 
For an integer $n\geq0$ and functions $v$ on $\bR^d$, whose 
partial derivatives up to order $n$ are functions, 
we use the notation $D^nv$ for the collection 
$
\{D_{\alpha}v:|\alpha|=n\}, 
$
and define
$$
|D^nv|^2=\sum_{|\alpha|=n}|D_{\alpha}v|^2. 
$$
For differentiable functions $v=(v^1,...,v^d):\bR^d\to \bR^d$ 
the notation $Dv$ means the Jacobian matrix 
whose $j$-th entry in the $i$-th row is $D_jv^i$. 

The space of smooth functions $\varphi=\varphi(x)$ with compact support on 
the $d$-dimensional Euclidean space $\bR^d$ is denoted by 
$C_0^{\infty}$. For $p, q\geq 1$ we denote by $\cL_{q}=\cL_{q}(Z, \bR)$ 
the Banach spaces of $\bR$-valued $\cZ$-measurable functions $h=h(z)$ of $z\in Z$
such that 
$$ 
|h|^{q}_{\cL_{q}}=\int_{\bR^d}|h(z)|^{q}\,\mu(dz)<\infty. 
$$
The notation 
$\cL_{p,q}$ means the space $\cL_{p}\cap\cL_{q}$ with the norm 
$$
|v|_{\cL_{p,q}}=\max(|v|_{\cL_{p}},|v|_{\cL_{q}}) \quad\text{for $v\in \cL_{p}\cap\cL_{q}$}.
$$
The space of sequences
$\nu=(\nu^{1},\nu^{2},...)$
of real numbers $\nu^{k}$ with finite norm
$$
|\nu|_{l_2}=\big(
\sum_{k=1}^{\infty}|\nu^k|^{2}\big)^{1/2}
$$
is denoted by $l_2$. 

The Borel $\sigma$-algebra of a separable 
Banach space $V$ is denoted by $\cB(V)$, and for $p\geq0$ 
the notations $L_p([0,T],V)$ and $L_p(\bR^d,V)$ are used 
for the space of $V$-valued Borel-measurable functions $f$ on $[0,T]$ and $g$ on 
$\bR^d$, respectively, such that $|f|_V$ and $|g|_V$ have finite Lebesgue integral 
over $[0,T]$ and $\bR^d$, respectively.  
For $p\geq 1$ and $g\in L_p(\bR^d,V)$ 
we use the notation $|f|_{L_p}$, defined by 
$$
|f|^p_{L_p}=\int_{\bR^d}|f(x)|_V^p\,dx<\infty.
$$
In the sequel, $V$ will be $\bR$, $l_2$ or $\cL_{p,q}$. 
For integer $n\geq 0$ the space of functions from $L_p(\bR^d,V)$, 
whose generalised derivatives up to order $n$ are also in $L_p(\bR^d,V)$, 
is denoted by $W^n_p=W^n_p(\bR^d,V)$ with the norm
$$
|f|_{W^n_p}:=\sum_{|\alpha|\leq n}|D_\alpha f|_{L_p}.
$$
By definition $W^0_p(\bR^d,V)=L_p(\bR^d,V)$. Moreover, 
we use $\bW^n_p=\bW^n_p(V)$ to denote the space of 
$\cP$-measurable functions mapping from $\Omega\times [0,T]$ 
into $W^n_p=W^n_p(\bR^d,V)$ such that 
$$
|f|^p_{\bW^n_p}:=E\int_0^T |f|^p_{W^n_p} \,dt<\infty,
$$ 
and we use $\bL_p$ to denote $\bW^0_p$. 
For $m\in\bR$ and $p\in(1,\infty)$ we use the notation $H^m_p=H^m_p(\bR^d;V)$ 
for the Bessel potential space 
with exponent $p$ and order $m$, defined as the space of $V$-valued
generalised functions  $\varphi$ on $\bR^d$ such that 
$$
(1-\Delta)^{m/2}\varphi\in L_p\quad\text{and}\quad 
|\varphi|_{H^m_p}:=|(1-\Delta)^{m/2}\varphi|_{L_p}<\infty,  
$$ 
where $\Delta=\sum_{i=1}^dD_i^2$. Moreover, we use $\bH^m_p$ to denote 
the space of $\cP$-measurable functions from $\Omega\times [0,T]$ to $H^m_p$ such that 
$$
|f|^p_{\bH^n_p}:=E\int_0^T |f_t|^p_{H^n_p} \,dt<\infty.
$$ 
We will often omit the target space $V$ in the notations $W^n_p(V)$, $H^m_p(V)$, $\bW^n_p(V)$ and $\bH^m_p(V)$
for convenience if $V=\bR$.
When $V=\cL_{p,q}$ we use $W^n_{p,q}$, $H^m_{p,q}$ $\bW^n_{p,q}$ and $\bH^m_{p,q}$ to denote 
$W^n_p(\cL_{p,q})$, $H^m_p(\cL_{p,q})$ $\bW^n_p(\cL_{p,q})$ and $\bH^m_p(\cL_{p,q})$ respectively, and we use $\bL_{p,q}$ to denote $\bW^0_{p,q}$.
\begin{remark}
If $V$ is a UMD space, see for example \cite{HNVW} for the definition of UMD spaces, 
then by Theorem 5.6.11 in \cite{HNVW} for $p>1$ and integers $n\geq 1$  we have
$
W^n_p(V)=H^n_p(V)
$
with equivalent norms. Clearly, $\cL_{p,q}$ is a UMD space for $p, q\in(1,\infty)$,  
which implies $W^n_{p,q}=H^n_{p,q}$ for non-negative integers $n$ and 
$p,q\in (1,\infty)$. 
\end{remark}

\mysection{Formulation of the results}

To formulate our assumptions we fix a constant $K$, a nonnegative number $m$, 
an exponent $p\in[2,\infty)$, and non-negative 
$\cZ$-measurable functions $\bar\eta$ and $\bar\xi$ on $Z$ such 
that they are bounded by $K$ and 
$$
K^2_{\eta}:=\int_{Z}|\bar{\eta}(z)|^2\,\mu(dz)<\infty
\quad 
K^2_{\xi}:=\int_{Z}|\bar{\xi}(z)|^2\,\nu(dz)<\infty. 
$$
We denote by $\lceil m \rceil$ the smallest integer which is greater than or equal to $m$, 
and $\lfloor m \rfloor$ the largest integer which is less than or equal to $m$. 

\begin{assumption}                                           \label{assumption L}
The derivatives in $x\in\bR^d$ of $a^{ij}$ up to order $\max\{\lceil m \rceil,2\}$, 
the derivatives of $b^i$ in $x$ up to order $\max\{\lceil m \rceil, 1\}$ 
and the derivatives of $c$ up to order $\lceil m \rceil$ are $\cP\otimes \cB(\bR^d)$-measurable 
functions on $\Omega\times H_T$, bounded by $K$ for all $i,j=1,2,...d$. 
The functions $\sigma^i=(\sigma^{ir})_{r=1}^\infty$ and $\beta=(\beta)_{r=1}^\infty$ 
and their derivatives up to order $\lceil m \rceil+1$ are $l_2$-valued 
$\cP\otimes \cB(\bR^d)$-measurable functions, bounded by $K$. 
Moreover $a^{ij}=a^{ji}$ for all $i,j=1,...d$, and for 
$P\otimes dt \otimes dx$-almost all $(\omega,t,x)\in\Omega\times H_T$
\begin{equation}                                          \label{parabolicity}
\alpha^{ij}_t(x)z^iz^j\geq 0\quad \text{for all $z=(z^1,...,z^d)$},
\end{equation}
 where
$$
\alpha^{ij}=2a^{ij}-\sigma^{ir}\sigma^{jr}.
$$
\end{assumption}

\begin{assumption}                                                   \label{assumption xi}
The mapping $\xi=(\xi^i)$ is an $\bR^d$-valued 
$\cP\otimes\cB(\bR^d)\otimes \cZ$-measurable function on 
$\Omega\times[0,T]\times\bR^d\times Z$.  
Its derivatives in $x\in \bR^d$ up to order $\max\{\lceil m \rceil,3\}$ exist 
and are continuous in $x\in \bR^d$ such that
$$
 |D^k\xi|\leq \bar{\xi}
\quad k=0,1,2,...,\max\{\lceil m \rceil, 3\}:=\bar m
$$
for all $(\omega,t,x,z)\in \Omega\times H_T\times Z$. Moreover, 
$$
K^{-1}\leq |\det(\bI+\theta D\xi_{t,z}(x))|
$$
for all $(\omega,t,x,z,\theta)\in \Omega\times H_T\times Z \times [0,1]$, 
where $\bI$ is the $d\times d$ identity matrix, 
and $D\xi$ denotes the Jacobian matrix of $\xi$ in $x\in\bR^d$.
\end{assumption}
\begin{assumption}                                                   \label{assumption eta}
The function $\eta=(\eta^i)$ maps $\Omega\times[0,T]\times\bR^d\times Z$ 
into $\bR^d$ such that Assumption \ref{assumption xi} holds with $\eta$ and 
$\bar\eta$ in place of $\xi$ and $\bar\xi$, respectively. 
\end{assumption}

\begin{assumption}                                              \label{assumption free}
The free data $f=(f_t)_{t\in[0,T]}$, $g=(g^r_t)_{t\in[0,T]}$  
and $h=(h_t)_{t\in[0,T]}$ 
are $\cP$-measurable processes 
with values in $H^m_p$, $H^{m+1}_p(l^2)$ and  
$H^{m+1}_{p,2}=H_p^{m+1}(\cL_{p,2})$, respectively,  
such that almost surely $\cK^p_{p,m}(T)<\infty$, where 
\begin{equation*}                            
\cK^p_{p,m}(t):=\int_0^t|f_s|^p_{H^m_p}+|g_s|^p_{H^{m+1}_p(l_2)}
+|h_s|^p_{H^{m+1}_{p,2}}+{\bf 1}_{p>2}|h_s|^p_{H^{m+2}_{p,2}} \,ds,  \quad t\leq T.
\end{equation*}
The initial value $\psi$ is an $\cF_0$-measurable 
random variable with values in $H^m_p$.
\end{assumption}

\begin{remark}                                                                    \label{remark ibp}
By Taylor's formula we have 
$$
v(x+\eta(x))-v(x)-\eta(x)\nabla v(x)
=\int_0^1\eta^k(x)(v_k(x+\theta\eta(x))-v_k(x))\,d\theta
$$
$$
=\int_0^1\eta^k(x)D_k(v(x+\theta\eta(x))-v(x))\,d\theta
-\int_0^1\theta\eta^k(x)\eta^l_k(x)v_l(x+\theta\eta(x))\,d\theta
$$
for every $v\in C_0^{\infty}$, where to ease notation 
we do not write the arguments $t$ and $z$ and write $v_k$ instead of $D_kv$ 
for functions $v$. 
Due to Assumption \ref{assumption eta} these equations extend to $v\in W^1_p$ 
for $p\geq2$ as well. 
Hence after changing the order of integrals, by integration by parts we 
obtain 
$$
(\cN^{\eta} v,\varphi)=-(\cJ_{\eta}^{k}v,D_k\varphi)+(\cJ_{\eta}^{0}v,\varphi)
$$
for $\varphi\in C_0^{\infty}$, 
with 
\begin{align}                                                       
\cJ^{k}_{\eta}(t)v(x)
=&\int_0^1\int_{Z}
\eta^k(v(\tau_{\theta\eta}(x))-v(x))\,\mu(dz)\,d\theta, \quad k=1,2,...,d,                \label{def Jk}\\
\cJ_{\eta}^{0}(t)v(x)
=&-\int_0^1\int_{Z}
\{\sum_k\eta^k_k(v(\tau_{\theta\eta}(x))-v(x))
+\theta\eta^k(x)\eta^l_k(x)v_l(\tau_{\theta\eta}(x))\}
\,\mu(dz)\,d\theta,                                                                                                     \label{def J0}
\end{align}
where for the sake of short notation   
the arguments $t,z$ of $\eta$ and $\eta_k$ have been omitted, and 
\begin{equation}                                                                                                       \label{tau0}
\tau_{\theta\eta}(x):=x+\theta\eta_{t,z}(x)\quad 
\text{for $x\in\bR^d$, $t\in[0,T]$, $z\in Z$ and 
$\theta\in[0,1]$}.  
\end{equation}
Operators $J^{k}_{\xi}$ and $J^{0}_{\xi}$ are defined 
as $J^{k}_{\eta}$ and  $J^{0}_{\eta}$
in  
\eqref{def Jk} and \eqref{def J0} but with 
$\xi$ everywhere in place of $\eta$.  
\end{remark}

\begin{definition}                           \label{definition solution}
An $L_p$-valued cadlag $\cF_t$-adapted process $u=(u_t)_{t\in[0,T]}$ 
is a {\it generalised solution} to equation \eqref{eq1} with initial value $u_0=\psi$, 
if $u_t\in W^1_p$  for $P\otimes dt$-almost every $(\omega, t)\in\Omega\times[0,T]$, 
such that $u\in L_p([0,T],W^1_p)$ almost surely, 
and for each $\varphi\in C^\infty_0(\bR^d)$
$$
(u_t,\varphi)=(\psi,\varphi)+\int_0^t\langle \cA_s u_s,\varphi \rangle
+(f_s,\varphi)\,ds
$$
\begin{equation}
 +\int_0^t(\cM_s^ru_s,\varphi)\,dw_s^r+ 
\int_0^t\int_Z
\int_{\bR^d}
(u_{s-}(x+\eta_{s,z}(x))-u_{s-}(x)+h_{s}(x,z))\varphi(x)\,dx \,\tilde{\pi}(dz,ds)            \label{equation}
\end{equation} 
for all $t\in[0,T]$ and almost all $\omega\in \Omega$, where
$$
\langle\cA_s u_s,\varphi\rangle
:= -(a_s^{ij}D_ju_s, D_i\varphi)+(\bar {b}_s^iD_iu_s+c_su_s,\varphi)
$$
$$
-(\cJ_{\xi}^{(i)}u_s,D_i\varphi)+(\cJ_{\xi}^{(0)}u_s,\varphi)
-(\cJ_{\eta}^{(i)}u_s,D_i\varphi)+(\cJ_{\eta}^{(0)}u_s,\varphi) 
$$
with $\bar b^i_s=b^i_s-D_ja^{ij}_s$ for all $s\in[0,T]$, and the stochastic 
integrals are It\^o integrals. 
\end{definition}

\begin{theorem}                                                                              \label{theorem main}
If Assumptions \ref{assumption L} through  
\ref{assumption free} hold with $m\geq 0$, then   
there is at most one generalised solution to \eqref{eq1}. 
If Assumptions \ref{assumption L} 
through \ref{assumption free} hold with $m\geq 1$,  
then there is a unique generalised solution 
$u=(u_t)_{t\in[0,T]}$, which is a weakly cadlag 
$W^m_p$-valued adapted process, and 
it is a strongly cadlag $W_p^s$-valued process for any $s\in[0,m)$. 
Moreover, 
\begin{equation}                                                \label{estimate main}
E\sup_{t\in[0,T]}|u_t|^q_{W^s_p}\,dt
\leq N\left(E|\psi|^q_{W^s_p}
+E\cK^q_{p,s}(T) \right)
\quad \text{for $s\in[0,m]$, $q\in(0,p]$}
\end{equation}
with a constant $N=N(d,m,p,q,T,K,K_{\xi},K_{\eta})$. 
\end{theorem}

\mysection{Preliminaries}

For vectors $v=(v^1,....,v^d)\in\bR^d$ we define the operators $T^v$, $I^v$ and $J^v$ by 
$$
T^v\varphi(x)=\varphi(x+v)-\varphi(x), 
\quad I^v\varphi(x)=\varphi(x+v)-\varphi(x), 
$$
\begin{equation}                                                                            \label{def TJI}
\quad J^v\phi(x)=\varphi(x+v)-\phi(x)-v^iD_i\phi(x)\quad x\in\bR^d
\end{equation}
acting on functions $\varphi$ and $\phi$ defined on $\bR^d$ such that the generalised 
derivative $D_i\phi$ exist. 
If $v=v(x)$ is a function of $x\in\bR^d$ then the notation 
$T^v$, $I^v$ and $J^v$ mean the operators defined by \eqref{def TJI} with 
$v(x)$ in place $v$. For example, $J^{\xi}$ and $J^{\eta}$ mean for each $\omega\in\Omega$, 
$t\in[0,T]$ and $z\in Z$ the operators defined 
on differentiable functions $\varphi$ on $\bR^d$ by 
\begin{align*}                                                                             
J^{\xi}\varphi(x)&=\varphi(x+\xi(x))-\varphi(x)-\eta^i(x)D_iv(x),        \nonumber\\
J^{\eta}\varphi(x)&=\varphi(x+\eta(x))-\varphi(x)-\eta^i(x)D_iv(x), \quad x\in\bR^d 
\end{align*}
for each fixed variable $(\omega,t,z)$ suppressed in this notation.
We will often use the Taylor formulas  
\begin{equation}                                                                     \label{taylor1}
I^{v}\varphi(x)=\int_0^1\varphi_i(x+\theta v)v^i\,d\theta
\end{equation}
and
\begin{equation}                                                                     \label{taylor2}  
J^v\phi(x)=\int_0^1(1-\theta)\phi_{ij}(x+\theta v)v^iv^j  
\,d\theta
\end{equation}
with  $\varphi_i:=D_i\varphi$ and $\phi_{ij}:=D_{i}D_{ j}\phi$, 
which hold for every $x\in\bR^d$ when $\varphi\in C^1(\bR^d)$ and $\phi\in C^2(\bR^d)$, 
and they hold for $dx$-almost every $x\in\bR^d$  
when $\varphi\in W^1_p$ and $\phi\in W^2_p$.

We fix a non-negative smooth function $k=k(x)$ 
with compact support on $\bR^d$ 
such that $k(x)=0$ 
for $|x|\geq 1$, $k(-x)=k(x)$ for $x\in \bR^d$, 
and $\int_{\bR^d}k(x)\,dx=1$. 
For $\varepsilon>0$ and locally integrable 
functions $v$ of  $x\in\bR^d$ we use the notation 
$v^{(\varepsilon)}$ for the mollification of $v$, defined by 
\begin{equation}                                                             \label{mollification}
v^{(\varepsilon)}(x):=S^\varepsilon v(x):=\varepsilon^{-d}
\int_{\bR^d}v(y)k((x-y)/\varepsilon)\,dy,\quad x\in \bR^d. 
\end{equation}
Note that if $v=v(x)$ is a locally Bochner-integrable function on $\bR^d$ 
taking values in a Banach space, the mollification of $v$ is defined 
as \eqref{mollification} in the sense of Bochner integral.

The following lemmas are taken from \cite{DGW} 
and for their proof we refer to \cite{DGW}. 
\begin{lemma}                                      \label{lemma TIJ}
Let Assumption \ref{assumption eta} hold.   
 Then for every $(\omega,t,z)\in\Omega\times[0,T]\times Z$ 
the operators $T^{\eta}$, $I^{\eta}$ 
and $J^{\eta}$ are bounded linear operators from $W^k_p$ to $W^k_p$, 
from $W^{k+1}_p$ to $W_p$ 
and from $W^{k+2}_p$ to $W^k_p$ respectively, for $k=0,1,...,m$, 
such that $T^{\eta}\varphi$, 
$I^{\eta}f$ 
and $J^{\eta}g$ are $\cP\otimes \cZ$-measurable 
$W^k_p$-valued functions of $(\omega, t,z)$,  and
$$
|T^{\eta}\varphi|_{W^k_p}\leq N|\varphi|_{W^k_p},
\quad
 |I^{\eta}f|_{W^k_p}\leq N\bar{\eta}(z)|f|_{W^{k+1}_p},
\quad 
|J^{\eta}g|_{W^k_p}\leq N\bar{\eta}^2(z)|g|_{W^{k+2}_p}
$$
for all $\varphi\in W^k_p$, $f\in W^{k+1}_p$ 
and $g\in W^{k+2}_p$, 
where $N$ is a constant 
only depending on $K,m,d,p$.
\end{lemma}
\begin{lemma}                                                                      \label{lemma diff}
Let $\rho$ be 
a $C^k(\bR^d)$-diffeomorphism for some $k\geq1$, such that 
\begin{equation}                                                                    \label{diffeo korlat}                                                      
M\leq |\det D\rho | 
\,\,
\text{and\,\,  $|D^l\rho|\leq N$ \quad for $l=1,2,...,k$}                                                       
\end{equation}
for some constants $M>0$ and $N>0$. Then
there are positive constants $M'=M'(N,d)$ and $N'=N'(N,M,d,k)$ 
such that \eqref{diffeo korlat} holds 
with $g:=\rho^{-1}$, the inverse of $\rho$, in place of $\rho$, 
with $M'$ and $N'$ in place of $M$ 
and $N$, respectively.
\end{lemma}
The following lemma is a slight generalisation of Lemma 3.4 in \cite{DGW}. 
\begin{lemma}                                                                      \label{lemma epsilon}
Let $\rho$ be 
a $C^k(\bR^d)$-diffeomorphism for $k\geq2$, such that 
\eqref{diffeo korlat} holds                                                     
for some positive constants $M$ and $N$. 
Then there is a positive constant $\varepsilon_0=\varepsilon_0(M,N,d)$ 
such that $\rho_{\varepsilon,\vartheta}:=\vartheta\rho+(1-\vartheta)\rho^{(\varepsilon)}$ 
is a $C^{k}(\bR^d)$-diffeomorphism 
for every $\varepsilon\in(0,\varepsilon_0)$ and $\vartheta\in[0,1]$,   
and \eqref{diffeo korlat} remains valid for $\rho_{\varepsilon,\vartheta}$ 
in place of $\rho$, with 
$M{''}=M/2$ in place of $M$. Moreover, $\rho^{(\varepsilon)}$ is a $C^{\infty}$-diffeomorphism 
for $\varepsilon\in(0,\varepsilon_0)$. 
\end{lemma}

\begin{proof}
We show first that $|\det D\rho_{\varepsilon,\vartheta}|$ is separated 
away from zero for sufficiently small 
$\varepsilon>0$. To this end observe that for bounded Lipschitz functions 
$v=(v^1,v^2,...,v^d)$ on $\bR^d$ and 
$v_{\varepsilon,\vartheta}:=\vartheta v+(1-\vartheta)v^{\varepsilon}$ 
we have  
$$
|\Pi_{i=1}^dv^i-\Pi_{i=1}^dv^{i}_{\varepsilon,\vartheta}|
\leq \sum_{i=1}^dK^{d-1}|v^i-v^{i}_{\varepsilon,\vartheta}|
\leq K^{d-1}L\varepsilon
\quad
\text{for any $\varepsilon>0$ and $\vartheta\in[0,1]$},  
$$
where 
$L$ is the Lipschitz constant of $v$ and $K$ is a bound for $|v|$. 
Using this observation and taking into account that $D_i\rho^l$ is bounded by $N$ 
and it is Lipschitz 
continuous with a Lipschitz constant not larger than $N$, we get 
$$
|\det D\rho-\det D\rho_{\varepsilon,\vartheta}|\leq d!\,N^{d}\varepsilon. 
$$
Thus setting $\varepsilon'=M/(2d!\,N^{d})$,  for $\varepsilon\in(0,\varepsilon')$ 
 and $\vartheta\in[0,1]$ we have 
$$
|\det D\rho_{\varepsilon,\vartheta}|
\geq 
|\det D\rho|-|\det D\rho-\det D\rho_{\varepsilon,\vartheta}|
$$
$$
\geq |\det D\rho|/2\geq M/2.
$$
Clearly, $\rho_{\varepsilon,\vartheta}$ is a $C^{k}$ function. 
Hence by the implicit function theorem $\rho_{\varepsilon,\vartheta}$ 
is a local $C^{k}$-diffeomorphism for $\varepsilon\in(0,\varepsilon')$ and $\vartheta\in[0,1]$.  
We prove now that $\rho_{\varepsilon,\vartheta}$ is a global $C^{k}$-diffeomorphism 
for sufficiently small $\varepsilon$. Since by the previous lemma $|D\rho^{-1}|\leq N'$,  
we have 
\begin{align*}
|x-y|\leq &N'|\rho(x)-\rho(y)|                             \\
\leq &N'|\rho_{\varepsilon,\vartheta}(x)-\rho_{\varepsilon,\vartheta}(y)|
+N'|\rho(x)-\rho_{\varepsilon,\vartheta}(x)+\rho_{\varepsilon,\vartheta}(y)-\rho(y)| 
\end{align*}
for all $x,y\in\bR^d$ and $\varepsilon>0$ and $\vartheta\in[0,1]$. Observe that
$$
|\rho(x)-\rho_{\varepsilon,\vartheta}(x)+\rho_{\varepsilon,\vartheta}(y)-\rho(y)|
\leq\int_{\R^d}|\rho(x)-\rho(x-\varepsilon u)+\rho(y-\varepsilon u)-\rho(y)|k(u)\,du
$$
$$
\leq 
\int_{\R^d}\int_0^1
\varepsilon|u||\nabla\rho(x-\theta\varepsilon u)-\nabla\rho(y-\theta\varepsilon u)|k(u)
\,d\theta\,du
$$
$$
\leq\varepsilon N|x-y|\int_{|u|\leq1}|u|k(u)\,du
\leq\varepsilon N|x-y|.
$$
Thus 
$|x-y|
\leq N'|\rho_{\varepsilon,\vartheta}(x)-\rho_{\varepsilon,\vartheta}(y)|
+\varepsilon N'N|x-y|$.  
Therefore  setting $\varepsilon''=1/(2NN')$, for all $\varepsilon\in(0,\varepsilon'')$  
and $\vartheta\in[0,1]$ we have 
\begin{equation*}
|x-y|\leq 2N'|\rho_{\varepsilon,\vartheta}(x)-\rho_{\varepsilon,\vartheta}(y)|
\quad\text{for all $x,y\in\bR^d$}, 
\end{equation*}
which implies $\lim_{|x|\to\infty}|\rho_{\varepsilon,\vartheta}(x)|=\infty$, i.e., 
under 
$\rho_{\varepsilon,\vartheta}$ the pre-image of any compact set is a compact set 
for each $\varepsilon\in(0,\varepsilon'')$ and $\vartheta\in[0,1]$. 
A continuous function with this 
property is called a {\it proper function}, and by Theorem 1 in \cite{G} a local $C^1$- 
diffeomorphism from $\bR^d$ into $\bR^d$ is a global diffeomorphism 
if and only if it is a proper function. 
Thus we have 
proved that $\rho_{\varepsilon,\vartheta}$ is 
a global $C^{k}$-diffeomorphism for 
$\varepsilon\in(0,\varepsilon_0)$ and $\vartheta\in[0,1]$,  
where $\varepsilon_0=\min(\varepsilon',\varepsilon'')$. 
Clearly, $\rho_{\varepsilon,0}=\rho^{(\varepsilon)}$ is a $C^{\infty}$ function and hence 
it is a $C^{\infty}$-diffeomorphism for every $\varepsilon\in(0,\varepsilon_0)$. 

Now we can complete the proof of the lemma by noting that since  
$D_{j}\rho^{(\varepsilon)}=(D_{j}\rho)^{(\varepsilon)}$,   
the condition $|D^i\rho|\leq N$ implies  
    $|D^i\rho_{\varepsilon,\vartheta}|\leq N$ for any $\varepsilon>0$ 
    and $\vartheta\in[0,1]$.  
\end{proof}

For fixed $\varepsilon>0$ and $\vartheta\in[0,1]$ let 
$\rho_{\varepsilon,\vartheta}$ denote 
any of the functions 
$$
\rho_{\varepsilon,\vartheta}(x):
=x+\vartheta\eta_{t,z}(x)
+(1-\vartheta)\eta_{t,z}^{(\varepsilon)}(x), 
\quad
\rho_{\varepsilon,\vartheta}(x):
=x+\vartheta\xi_{t,z}(x)
+(1-\vartheta)\xi_{t,z}^{(\varepsilon)}(x)
\quad x\in\bR^d. 
$$
for each $(\omega,t,z)\in\Omega\times[0,T]\times Z$, 
and assume that Assumptions \ref{assumption xi} and 
\ref{assumption eta} hold. 
Then by the inverse function theorem 
$\rho$ is a local $C^1(\bR^d)$-diffeomor\-phism for each $t$, $\theta$ and $z$.  
Since
$$
|\eta_{t,z}(x)|\leq \bar\eta(z)<\infty,  \quad 
|\xi_{t,z}(x)|\leq \bar\xi(z)<\infty, 
$$
we have 
$
\lim_{|x|\to\infty}|\rho_{\varepsilon,\vartheta}(x)|=\infty. 
$ 
Hence $\rho_{\varepsilon,\vartheta}$ is a global $C^1$-diffeomorphism on $\bR^d$, 
for $\varepsilon>0$, $\vartheta\in[0,1]$, for 
each $t\in[0,T]$, $z\in Z$ and $\theta\in[0,1]$,  
by Theorem 1 in \cite{G}. 
Note that by the formula on the derivative of inverse functions  
a $C^{1}(\bR^d)$-diffeomorphism and its inverse have continuous derivatives 
up to the same order. Thus Lemmas \ref{lemma diff} and \ref{lemma epsilon} 
imply the following lemma, which is a slight generalisation of Corollary 3.6 
in \cite{DGW}. 

\begin{lemma}                                                                  \label{lemma diff2}
Let Assumptions \ref{assumption eta} and \ref{assumption xi} hold. 
Then there is a positive constant
 $\varepsilon_0=\varepsilon_0(K,d)$ 
such that $\rho=\rho_{\varepsilon,\vartheta}$ is a $C^k$-diffeomorphism on 
$\bR^d$ for $k=\bar m$, for any 
$\varepsilon\in(0,\varepsilon_0)$, $\vartheta \in[0,1]$ 
and $(\omega,t,z)\in\Omega\times[0,T]\times Z$. 
Moreover, for some constants 
$M=M(K,d,\bar m)$ and $N=(K,d,\bar m)$
\begin{equation}                                                           \label{smootheta}
M\leq  
\min(|{\rm{det}}D\rho|,|{\rm{det}}(D\rho)^{-1}|), \quad 
\max(|D^k\rho|, |D^k \rho^{-1}|)\leq N
\end{equation}
for any 
$\varepsilon\in(0,\varepsilon_0)$, $\vartheta \in[0,1]$, 
$(\omega,t,z)\in\Omega\times[0,T]\times Z$ and for 
$k=1,2,...,\bar m$. Furthermore, if $\vartheta=0$ then $\rho$ is a 
$C^{\infty}$-diffeomorphism for each $\varepsilon\in(0,\varepsilon_0)$, 
$(\omega,t,z)\in\Omega\times[0,T]\times Z$, and for each integer $m\geq1$ 
there are constants $M=M(K,d,m)$ and $N=N(K,d,m)$ such the estimates in  
\eqref{smootheta} hold for all $\varepsilon\in(0,\varepsilon_0)$, $\vartheta \in[0,1]$, 
$(\omega,t,z)\in\Omega\times[0,T]\times Z$ and  
$k=1,2,...,m$. 
\end{lemma}

\begin{lemma}                                                               \label{lemma e}
Let $V$ be a separable Banach space, 
and let $f=f(x)$ be a $V$-valued function of $x\in\bR^d$ such that $f\in L_p(V)=L_p(\bR^d,V)$ 
for some $p\geq 1$. Then we have
$$
|f^{(\varepsilon)}|_{L_p(V)}\leq |f|_{L_p(V)}\quad 
\text{and} 
\quad \lim_{\varepsilon\to 0}|f^{(\varepsilon)}-f|_{L_p(V)}=0.
$$
\end{lemma}
\begin{proof}
This lemma is well-known. Its proof can be found, e.g., in \cite{GW2019}, see Lemma 4.4 
therein.
\end{proof}

Recall that $L_p(\cL_{q_1}\cap\cL_{q_2})$ denotes the $L_p$-space of 
$\cL_{q_1}\cap\cL_{q_2}$-valued functions on $\bR^d$ with respect to the Lebesgue 
measure on $\bR^d$. 
Since $(Z,\cZ,\mu)$ is a $\sigma$-finite separable measure space, 
$\cL_{q_1}\cap\cL_{q_2}$ is a separable Banach space for any $q_1,q_2\in[1,\infty)$.  
Hence, by Lemma 3.6 in \cite{GW2019} for each $v\in L_p(\cL_{q_1}\cap\cL_{q_2})$, $p>1$ 
there is a $\cB\otimes\cZ$-measurable function $\bar v=\bar v(x,z)$ such that 
for every $x\in\bR^d$ we have $v(x,z)=\bar v(x,z)$ for $\mu$-almost every $z\in Z$. 
Therefore if $v\in L_p(\cL_{q_1}\cap\cL_{q_2})$ 
for some $p>1$ then $v$ we may assume that 
it is $\cB(\bR^d)\otimes Z$-measurable real-valued function. 
Moreover, we will often use the following characterisation of 
$W^n_p(\cL_{q_1}\cap\cL_{q_2})$. 
\begin{lemma}
Let $v\in L_p(\cL_{p}\cap\cL_{q})$  for some $p,q\in(1,\infty)$, and let 
$\alpha$ be a multi-index. Then the following statements hold. 
\begin{enumerate}[(i)]
\item If $v_{\alpha}$, the $\cL_{p}\cap\cL_{q}$-valued generalised 
$D_{\alpha}$-derivative belongs to $L_p(\cL_p\cap\cL_q)$, then for $\mu$-almost every $z\in Z$ 
the function $v_{\alpha}(\cdot,z)$ belongs to $L_p(\bR^d,\bR)$ and it is the generalised 
$D_{\alpha}$-derivative of $v(\cdot,z)$. 
\item If $v_{\alpha}(\cdot,z)$, the generalised $D_{\alpha}$-derivative of the 
function $v(\cdot,z)$ belongs to $L_p(\bR^d,\bR)$ for 
$\mu$-almost every $z\in Z$ such that 
\begin{equation}                                                       \label{v}                                                             
\int_{\bR^d}
\Big(\int_{Z}|v_{\alpha}(x,z)|^{r}\,\mu(dz)\Big)^{p/r}\,dx<\infty 
\quad\text{for $r=p,q$}, 
\end{equation}
then $v_{\alpha}$ belongs to $L_p(\cL_p\cap\cL_q)$, and it is 
the $\cL_p\cap\cL_q$-valued generalised  $D_{\alpha}$-derivative of $v$. 
\end{enumerate}
\end{lemma}

\begin{proof}
(i) Let $\bar v_{\alpha}$ denote the $\cL_{p}\cap\cL_{q}$-valued 
generalised $D_{\alpha}$-derivative of 
$v$. 
Then 
\begin{equation*}                                             
\int_{\bR^d}\bar v_{\alpha}(x)\varphi(x)\,dx
=(-1)^n\int_{\bR^d}v(x)D_{\alpha}\varphi(x)\,dx
\end{equation*}
for every $\varphi\in C_0^{\infty}$, 
where the integrals are understood as Bochner integrals 
of $\cL_{p}\cap\cL_{q}$-valued functions. Hence 
$$
\int_{\bR^d}\int_{Z}\bar v_{\alpha}(x,z)\psi(z)\varphi(x)\mu(dz)\,dx
=(-1)^{|\alpha|}\int_{\bR^d}
\int_{Z}\bar v(x,z)\psi(z)\varphi_{\alpha}(x)\mu(dz)\,dx
$$
for all $\varphi\in C_0^{\infty}$ and bounded 
$\cZ$-measurable functions $\psi$ supported 
on sets of finite $\mu$-measure. 
We can use Fubini's theorem  to 
get 
$$
\int_{Z}\int_{\bR^d}\bar v_{\alpha}(x,z)\varphi(x)\,dx\,\psi(z)\mu(dz)
=(-1)^{|\alpha|}\int_{Z}\int_{\bR^d}v(x,z)\varphi_{\alpha}(x)\,dx\,\psi(z)\,\mu(dz). 
$$
Thus for each $\varphi\in C_0^{\infty}$ we have 
\begin{equation}                                            \label{simple_derivatives}
\int_{\bR^d}\bar v_{\alpha}(x,z)\varphi(x)\,dx
=(-1)^{|\alpha|}\int_{\bR^d}v(x,z)\varphi_{\alpha}(x)\,dx 
\end{equation}
for $\mu$-almost every $z\in Z$. 
Consequently, for $\mu$-almost every $z\in Z$ equation \eqref{simple_derivatives} 
holds for all $\varphi\in\Phi$ for a separable dense set 
$\Phi\subset C_0^{\infty}$ in $L_{p/(p-1)}(\bR^d,\bR)$. 
Notice that for $\mu$-almost every $z\in Z$ the functions $\bar v_{\alpha}(\cdot,z)$ 
and $v(\cdot,z)$ belong to $L_{p}(\bR^d,\bR)$.  
Hence there is a set $S\subset Z$ of 
full $\mu$-measure such that for $z\in S$ equation \eqref{simple_derivatives}  
holds for all $\varphi\in C_0^{\infty}$, 
which proves that for $z\in S$ the function $\bar v_{\alpha}(\cdot,z)$ 
is generalised $D_{\alpha}$-derivative 
of $v(\cdot, z)$.  
To prove (ii) notice that if for $\mu$-almost 
every $z\in Z$ the function $v_{\alpha}(\cdot,z)$ belongs to $L_p(\bR^d,\bR)$ 
and it is the $D_{\alpha}$ generalised derivative of the function $v(\cdot,z)$, then 
for $\mu$-almost every $z\in Z$ we have 
$$
\int_{\bR^d}v_{\alpha}(x,z)\varphi(x)\,dx
=(-1)^{|\alpha|}\int_{\bR^d}v(x,z)\varphi(x)\,dx
$$
for every $\varphi\in C_0^{\infty}$. Using condition \eqref{v} 
and that $v\in L_p(\cL_p\cap\cL_q)$,  
it is easy to check that, as functions of $z$, both sides of the a
bove equation are functions 
in $\cL_p\cap\cL_q$, and hence that these integrals define the same functions  
as the corresponding 
$\cL_p\cap\cL_q$-valued Bochner integrals.  This proves that $v_{\alpha}$ 
is the $\cL_{p}\cap\cL_q$-valued generalised $D_{\alpha}$-derivative of $v$. 
\end{proof}

\begin{lemma}                                                                  \label{lemma intIJ}
Let Assumptions \ref{assumption eta} hold with $m=0$. Then 
the following statements hold. 
\begin{enumerate}[(i)]
\item Let $\zeta$ be a $\cF\otimes\cB([0,T]\times\bR^d)\otimes\cZ$-measurable function 
on $\Omega\times H_T\times Z$ such that it is continuously differentiable 
in $x\in\bR^d$ and 
\begin{equation}                                                \label{assumption intI}
|\zeta|+|D\zeta|\leq K\bar\eta
\quad 
\text{for all $(\omega,t,x,z)\in\Omega\times H_T\times Z$}. 
\end{equation}
Then there is a constant $N=N(K,d)$ such that 
for $\varphi\in W^1_1$
\begin{equation}                                                             \label{intI}
A:=\int_{\bR^d}\zeta(t,x,z)I^{\eta_{t,z}}\varphi(x)\,dx
\leq N\bar\eta^2(z)\,|\varphi|_{L_1}
\quad 
\text{for all $(\omega,t,z)\in\Omega\times[0,T]\times Z$}.  
\end{equation}  
\item There is a constant $N=N(K,d)$ such that 
for all $\phi\in W^2_1$
\begin{equation}                                                             \label{intJ}
B:=\int_{\bR^d}J^{\eta_{t,z}}\phi(x)\,dx
\leq N\bar\eta^2(z)|\phi|_{L_1}.    
\end{equation}
\item
There is a constant $N=N(K,d)$ such that 
for all $\phi\in W^1_1$
\begin{equation}                                                             \label{C}
C:=\int_{\bR^d}I^{\eta_{t,z}}\phi(x)\,dx
\leq N\bar\eta(z)|\phi|_{L_1}.    
\end{equation}

\end{enumerate}  
\end{lemma}

\begin{proof}
The proof of \eqref{intJ} and \eqref{C} is given in \cite{Da} and \cite{DGW}. 
For the convenience of the reader we prove each of the above estimates here. 
We may assume that $\varphi, \phi\in C_0^{\infty}$.  
For each $(\omega,t,z,\theta)\in\Omega\times[0,T]\times Z\times [0,1]$ 
let $\tau^{-1}_{t,z,\theta}$ denote 
the inverse of the function 
$x\to x+\theta\eta_{t,z}(x)$. 
Using \eqref{taylor1} and \eqref{taylor2} by change of variables we have 
\begin{equation}                                                                               \label{iIJ}
A
=\int_0^1\int_{\bR^d}\nabla\varphi(x)\chi_{t,z,\theta}(x)\,dx\,d\theta, \quad                                   
B
=\int_0^1\int_{\bR^d}(1-\theta)D_{ij}\phi(x)\varrho^{ij}_{t,z,\theta}(x)\,dx\,d\theta                
\end{equation}
\begin{equation}                                                                               \label{IC}
C=\int_0^1\int_{\bR^d}\nabla\phi(x)\kappa_{t,z,\theta}(x)\,dx\,d\theta
\end{equation}
with
\begin{equation*}
\chi_{t,\theta,z}(x):=(\zeta\eta)(\tau^{-1}_{t,z,\theta}(x))
|{\rm{det}}D\tau^{-1}_{t,z,\theta}(x)|,\quad
\varrho^{ij}_{t,z,\theta}(x)
:=(\eta^i_{t,z}\eta^j_{t,z})(\tau^{-1}_{t,z,\theta}(x)) |{\rm{det}}D\tau^{-1}_{t,z,\theta}(x)| 
\end{equation*}
and 
$$
\kappa_{t,z,\theta}(x):=\eta(\tau^{-1}_{t,z,\theta}(x))
|{\rm{det}}D\tau^{-1}_{t,z,\theta}(x)|. 
$$
Due to \eqref{assumption intI} and Assumption \ref{assumption eta} 
we have a constant $N=N(K,d)$ such that 
$$
|D\chi_{t,\theta,z}(x)|\leq N\bar\eta^2(z),
\quad 
|D_{ij}\varrho^{ij}_{t,z,\theta}(x)|\leq N\bar\eta^2(z)
\quad
\text{and}\quad
|D\kappa_{t,\theta,z}(x)|\leq N\bar\eta(z)
$$
for all $(\omega,x,t,z,\theta)\in\Omega\times\bR^d\times[0,T]\times Z\times[0,1]$. 
Thus from \eqref{iIJ} and \eqref{IC}
by integration by parts 
we get \eqref{intI}, \eqref{intJ} and \eqref{C}.
\end{proof}

Next we present two important It\^o's formulas from \cite{GW2019} 
for the $p$-th power 
of the $L_p$-norm of a stochastic process.                                                       

\begin{lemma}                                                    \label{lemma Ito1}
Let $(u^i_t)_{t\in0,T}$ be a progressively measurable $L_p$-valued 
process such that there exist $f^i\in \bL_p$, 
$g^i=(g^{ir})_{r=1}^\infty\in\bL_p$, $h^i\in\bL_{p,2}$, and    
an $L_p$-valued 
$\mathcal{F}_0$-measurable random variable 
$\psi^i$ for each $i=1,2,...,M$ for some integer $M$, 
such that 
for every $\varphi\in C_0^{\infty}$
\begin{equation}                                                          \label{simple weak equation}
(u^i_t,\varphi)
=(\psi,\varphi)
+\int_0^t(f^i_s,\varphi)\,ds
+\int_0^t(g_s^{ir},\varphi)\,dw_s^r
+\int_0^t\int_Z(h^i_s,\varphi)\,\tilde{\pi}(dz,ds)
\end{equation}
for $P\otimes dt$-almost every $(\omega,t)\in\Omega\times[0,T]$ and all $i=1,2,...,M$.                                                       
Then there are 
$L_p$-valued adapted cadlag processes
 $\bar u=(\bar u^1,\bar{u}^2,...,\bar{u}^M)$ 
such that equation \eqref{simple weak equation}, with $\bar u^i$ in place of $u^i$, holds 
for every $i=1,2,...,M$ and each $\varphi\in C_0^{\infty}$  almost surely for all $t\in[0,T]$. 
Moreover,  $u^i=\bar u^i$ for $P\otimes dt$-almost every 
$(\omega,t)\in\Omega\times[0,T]$, and 
\begin{align}
|\bar u_t|^p_{L_p}
&= |\psi|_{L_p}^p
+p\int_0^t\int_{\mathbb{R}^d}|\bar u_s|^{p-2}\bar u^i_sg^{ir}_s\,dx\, dw^r_s                    \nonumber\\
& 
+\tfrac{p}{2}\int_0^t\int_{\mathbb{R}^d}\big( 2|\bar u_s|^{p-2}\bar u^i_sf^i_s
+(p-2)|\bar{u}_s|^{p-4}|\bar u^i_sg^{i\cdot}_s|_{l_2}^2
+|\bar u_s|^{p-2}\sum_{i=1}^M|g_s^{i\cdot}|_{\l_2}^2\big)\,dx\,ds                                                                \nonumber\\
& 
+p\int_0^t\int_Z\int_{\mathbb{R}^d}|\bar u_{s-}|^{p-2}\bar u^i_{s-}h^i_s
\,dx\,\tilde{\pi}(dz,ds)                                                                                              \nonumber\\
&
+\int_0^t\int_Z\int_{\mathbb{R}^d}
(|\bar u_{s-}+h_s|^p-|\bar u_{s-}|^p-p|\bar u_{s-}|^{p-2}\bar u^i_{s-}h^i_s)
\,dx\,\pi(dz,ds)                                                                                                         \nonumber
\end{align}
holds almost surely for all $t\in[0,T]$.  
\end{lemma}

\begin{lemma}                                                 \label{Ito Lp formula}
Let $u=(u_{t})_{t\in0,T}$ be a progressively measurable $W^1_p$-valued 
process such that the following conditions hold:\\
(i)$$
E\int_0^T|u_t|^p_{W^1_p}\,dt<\infty\,; 
$$
(ii) there exist $f\in \bL_p$ 
for $\alpha\in\{0,1,...,d\}$, 
$g=\in\bL_p$, $h\in\bL_{p,2}$, and    
an $L_p$-valued 
$\mathcal{F}_0$-measurable random variable $\psi$, 
such that 
for every  $\varphi\in C_0^{\infty}$ we have 
\begin{equation}                                                             \label{weak equation}                                   
(u_t,\varphi)=(\psi,\varphi)+\int_0^t(f_s^\alpha,D^*_\alpha\varphi)\,ds
+\int_0^t(g_s^r,\varphi)\,dw_s^r+\int_0^t\int_Z(h_s(z),\varphi)\,\tilde{\pi}(dz,ds)
\end{equation}
for $P\otimes dt$-almost every $(\omega,t)\in\Omega\times[0,T]$, where 
$D_{\alpha}^{\ast}=-D_{\alpha}$ for $\alpha=1,2,..,d$,  and $D^{\ast}_{\alpha}$ is the identity operator 
for $\alpha=0$.  
Then there is an $L_p$-valued adapted c\`adl\`ag process
 $\bar u=(\bar u_t)_{t\in[0,T]}$ 
such that for each $\varphi\in C_0^{\infty}$ equation \eqref{weak equation} holds 
with $\bar u$ in place of $u$ almost surely for all $t\in[0,T]$. 
Moreover,  $u=\bar u$ for $P\otimes dt$-almost every 
$(\omega,t)\in\Omega\times[0,T]$, and 
almost surely 
$$ 
|\bar u_t|^p_{L_p}= 
 |\psi|_{L_p}^p+p\int_0^t\int_{\mathbb{R}^d}|u_s|^{p-2}u_sg_s^r\,dx\,dw^r_s
 $$
 $$
+\tfrac{p}{2}\int_0^t\int_{\mathbb{R}^d}\big(2|u_s|^{p-2}u_sf^0_s-2(p-1)|u_s|^{p-2}f^i_sD_iu_s
+(p-1)|u_s|^{p-2}|g_s|^2_{l_2}   \big)\,dx\,ds
$$
$$
+\int_0^t\int_Z\int_{\mathbb{R}^d}p|\bar u_{s-}|^{p-2}\bar u_{s-}h_s\,dx\,\tilde{\pi}(dz,ds)
$$
\begin{equation*}                                                           
+\int_0^t\int_Z\int_{\mathbb{R}^d}\big( | \bar u_{s-}
+h_s|^p-| \bar u_{s-}|^p-p| \bar u_{s-}|^{p-2}\bar u_{s-}h_s  \big)\,dx\,\pi(dz,ds)   
\end{equation*}
for all $t\in[0,T]$, where $\bar u_{s-}$ denotes the left-hand limit in $L_p(\bR^d)$ 
of $\bar u$ at $s\in(0,T]$. 
\end{lemma}

The following slight generalisation of 
Lemma from \cite{GK2003} will play an essential role 
in obtaining supremum estimates. 
\begin{lemma}
                                            \label{lemma sup}
Let $T\in[0,\infty]$ and let $f=(f_t)_{t\geq0}$ 
and $g=(g_{t})_{t\geq0}$ be  nonnegative 
$\cF_{t}$-adapted  processes such that $f$ is 
a cadlag and $g$ is a continuous process. Assume 
\begin{equation}
                                                              \label{7.19.3}
E f_{\tau}{\bf1}_{g_{0}\leq c} \leq  Eg_{\tau}{\bf1}_{g_{0}\leq c}
\end{equation}
for any constant $c>0$ and bounded stopping
time $\tau\leq T$. Then,
for any bounded stopping
time $\tau\leq T$, for $\gamma\in(0,1)$
\begin{equation*}
E \sup_{t\leq \tau}f^{\gamma}_{t}
\leq\tfrac{2- \gamma}{1-\gamma}
E \sup_{t\leq \tau}g^{\gamma}_{t}.
\end{equation*}
\end{lemma}
{\em Proof\/}. This lemma is proved in 
\cite{GK2003} when both processes $f$ and $g$ are continuous. 
A word by word repetition of the proof in \cite{GK2003} extends it 
to the case when $f$ to be cadlag. For the convenience 
of the reader we present the proof below. 
By replacing $f_{t}$ and $g_{t}$
with $f_{t\wedge T}$ and $g_{t\wedge T}$,
respectively, we see that we may assume that $T=\infty$.
Then we replace $g_{t}$ with $\max_{s\leq t}g_{s}$
and see that without losing generality
we may assume that $g_{t}$ is nondecreasing.
In that case
fix a constant $c>0$ and let $\theta_{f}=\inf\{t\geq0:
f_{t}\geq c\}$, $\theta_{g}=\inf\{t\geq0:
g_{t}\geq c\}$. Then
$$
P(\sup_{t\leq\tau}f_{t}>c )\leq
      P(\theta_{f}\leq\tau )
\leq  P(\theta_{g}\leq\tau )
+P(\theta_{f}\leq\tau\wedge\theta_{g},\theta_{g}>\tau  )
$$
$$
\leq P(g_{\tau}\geq c )+P(g_{0}\leq c,
f_{\tau\wedge\theta_{g}\wedge\theta_{f}}\geq c )
\leq P(g_{\tau}\geq c )+\frac{1}{c}\,
E I_{g_{0}\leq c}f_{\tau\wedge\theta_{g}\wedge\theta_{f}} .
$$
In the light of (\ref{7.19.3})
we replace the  expectation with
$$
EI_{g_{0}\leq c} g_{\tau\wedge\theta_{g}\wedge\theta_{f}}
\leq E I_{g_{0}\leq c}g_{\tau\wedge\theta_{g} }
= E I_{g_{0}\leq c}(g_{\tau}\wedge g_{\theta_{g}})
$$
$$
\leq E I_{g_{0}\leq c}(g_{\tau}\wedge c) \leq
E(g_{\tau}\wedge c).
$$
Hence
$$
P(\sup_{t\leq\tau}f_{t}>c )
\leq P(g_{\tau}\geq c )+\frac{1}{c}\,
E(c\wedge g_{\tau })
$$
Now it only remains to substitute $c^{1/\gamma}$
in place of $c$ and integrate with respect to $c$
over $(0,\infty)$. The lemma is proved.

Finally we present a slight modification of Lemma 5.3 from \cite{DGW} which 
we will use in proving regularity in time of solutions 
to \eqref{eq1}-\eqref{ini1}.
\begin{lemma}                                                                   \label{lemma w}
Let $V$ be a reflexive Banach space, 
embedded continuously and densely into a Banach space $U$.  
Let $f$ be a $U$-valued weakly cadlag function on $[0,T]$ such that 
the weak limit in $U$ at $T$ from the left is $f(T)$. 
Assume 
there is a dense subset $S$ of $[0,T]$ such that $f(s)\in V$ for $s\in S$ 
and $\sup_{s\in S}|f(s)|_V<\infty$. 
Then $f$ is a $V$-valued function, which is cadlag 
in the weak topology of $V$, and 
hence $\sup_{s\in[0,T]}|f(s)|_V=\sup_{s\in S}|f(s)|_V$.
\end{lemma}

\begin{proof}
Since $S$ is dense in $[0,T]$, for a given $t\in[0,T)$ there is a sequence 
$\{t_n\}_{n=1}^{\infty}$ with elements in $S$ such that $t_n\downarrow t$. 
Due to $\sup_{n\in\N}|f(t_n)|_V<\infty$ and the reflexivity of $V$ there is  a subsequence $\{t_{n_k}\}$ 
such that $f(t_{n_k})$ converges weakly in $V$ to some element $v\in V$. 
Since $f$ is weakly cadlag in $U$, for every continuous linear functional $\varphi$ over $U$ 
we have $\lim_{k\to\infty}\varphi(f(t_{n_k}))=\varphi({f}(t))$. 
Since the restriction of $\varphi$ in $V$ is a continuous 
functional over $V$ we have $\lim_{k\to\infty}\varphi(f(t_{n_k}))=\varphi(v)$. 
Hence $f(t)=v$, which proves that $f$ is a $V$-valued function over $[0,T)$. Moreover,   
by taking into account that 
$$
|f(t)|_V=|v|_V\leq\liminf_{k\to\infty}|f(t_{n_k})|_V
\leq \sup_{t\in S}|f(t)|_V<\infty,
$$ we obtain $K:=\sup_{t\in[0,T)}|f(s)|_V<\infty$.
Let $\phi$ be a continuous linear functional over $V$. Due to the reflexivity of $V$, the dual $U^*$ 
of the space $U$ is densely embedded into $V^*$, the dual of $V$. 
Thus for $\phi\in V^{\ast}$ and $\varepsilon>0$ there is 
$\phi_\varepsilon\in U^*$ such that $|\phi-\phi_\varepsilon|_{V^*}\le \varepsilon$. Hence 
for arbitrary sequence $t_n\downarrow t$, $t_n\in[0,T]$ we have 
$$
|\phi(f(t))-\phi(f(t_n))|\leq |\phi_\varepsilon(f(t)-f(t_n))|+|(\phi-\phi_\varepsilon)(f(t)-f(t_n))|
$$
$$
\leq|\phi_\varepsilon(f(t)-f(t_n))|+\varepsilon|f(t)-f(t_n)|_V
\leq |\phi_\varepsilon(f(t)-f(t_n))|+2\varepsilon K.
$$
Letting here $n\to\infty$ and then $\varepsilon\to0$, we get
$$
\limsup_{n\to\infty}|\phi(f(t))-\phi(f(t_n)) |\leq0,  
$$
which proves that $f$ is right-continuous in the weak topology in $V$. 
We can prove in the same way that at each $t\in[0,T]$ the function $f$ 
has weak limit in $V$ from the left at each $t\in(0,T]$, which finishes the proof 
of the lemma. 
\end{proof}

\section{Some results on interpolation spaces}                              \label{section interpolation}

A pair of complex Banach spaces $A_0$ and $A_1$, which are 
continuously embedded into a Hausdorff  topological vector space $\cH$, is called an interpolation 
couple, and 
$A_{\theta}=[A_0,A_1]_{\theta}$ denotes the complex interpolation space   
between $A_0$ and $A_1$ with parameter $\theta\in(0,1)$. For an interpolation couple 
$A_0$ and $A_1$ the notations $A_0\cap A_1$ and $A_0+A_1$ is used for the subspaces 
$$
A_0\cap A_1=\{v\in\cH:v\in A_0\,\text{and}\,v\in A_1\}, 
\quad A_0+A_1=\{v\in\cH: v=v_1+v_2,\,\, v_i\in A_i\}
$$
equipped with the norms  $|v|_{A_0\cap A_1}=\max(|v|_{A_0},|v|_{A_1})$ and 
$$
|v|_{A_0+A_1}:=\inf\{|v_0|_{A_0}+|v_1|_{A_1}
:v=v_0+v_1,v_0\in A_0, \,v_1\in A_1\},  
$$
respectively. 
Then the following theorem lists some well-known 
facts about complex interpolation, see e.g., 1.9.3, 1.18.4 and 2.4.2 in \cite{T}
and 5.6.9 in \cite{HNVW}.

\begin{theorem}                                             \label{theorem i1}
\begin{enumerate}[(i)]
\item 
If $A_0,A_1$ and $B_0,B_1$ are two interpolation couples and 
$S:A_0+A_1\to B_0+B_1$ is a linear operator such that 
its restriction onto $A_i$ is a continuous operator into $B_i$ 
with operator norm $C_i$ 
for $i=0,1$, then its restriction onto $A_{\theta}=[A_0,A_1]_{\theta}$ 
is a continuous operator into 
$B_{\theta}=[B_0,B_1]_{\theta}$ with operator norm 
$C_0^{1-\theta}C_1^{\theta}$ for every $\theta\in(0,1)$.  
\item 
For a $\sigma$-finite measure space $\mathfrak M$ and 
an interpolation couple 
of separable Banach spaces $A_0$, $A_1$ we have 
$$
[L_{p_0}(\mathfrak M,A_0),L_{p_1}(\mathfrak M,A_1)]_{\theta
}=L_{p}(\mathfrak M,[A_0,A_1]_{\theta}),
$$ 
for every $p_0,p_1\in[1,\infty)$, $\theta\in(0,1)$, 
where $1/p=(1-\theta)/p_0+\theta/p_1$.
\item Let $H^m_p$ denote the Bessel potential spaces of complex-valued functions.  
Then 
for $m_0,m_1\in\bR$ and  $1<p_0,p_1<\infty$ 
$$
[H^{m_0}_{p_0},H^{m_1}_{p_1}]_{\theta}=H^m_p, 
$$
where $m=(1-\theta)m_0+\theta m_1$, 
and $1/p=(1-\theta)/p_0+\theta/p_1$. Moreover, for integers $m$ one has 
$H^m_p=W^m_p$ with equivalent norms. 
\item For a UMD Banach space $V$, denote by $H^m_p(V)$ 
the Bessel potential spaces of $V$-valued functions. 
Then for $1<p<\infty$ and $m_0, m_1\in \bR$
$$
[H^{m_0}_p(V),H^{m_1}_p(V)]_\theta=H^m_p(V)
$$
for every $\theta\in(0,1)$, where $m=(1-\theta)m_0+\theta m_1$.
\item For $\theta\in[0,1]$ there is a constant 
$c_{\theta}$ such that 
$$
|v|_{A_{\theta}}\leq c_{\theta}|v|_{A_0}^{1-\theta}|v|^{\theta}_{A_1}
$$
for all $v\in A_0\cap A_1$. 
\end{enumerate}
\end{theorem} 

We will also use the following theorem on the interpolation spaces between 
the interpolation couple 
$\cL_q\cap\cL_{p_0}$ and $\cL_q\cap\cL_{p_1}$,  
for $1\leq p_0\leq p_1$ and a fixed $q\notin(p_0,p_1)$, where
the notation $\cL_p$ means the $L_p$-space of real functions on a measure space 
$(Z,\cZ,\mu)$ with a $\sigma$-finite measure $\mu$ on a $\sigma$-algebra $\cZ$.  

\begin{theorem}                                                          \label{theorem Lpq interpolation}
For any 
$1\leq p_0\leq p_1\leq\infty$, $q\geq 1$ and $q\notin(p_0,p_1)$
we have 
$$
[\cL_q\cap\cL_{p_0},\cL_q\cap\cL_{p_1}]_{\theta}=\cL_q\cap\cL_{p}
$$ 
with equivalent norms for each $\theta\in(0,1)$, where 
$p$ is defined by $1/p=(1-\theta)/p_0+\theta/p_1$. 
\end{theorem}
This theorem is proved in \cite{R2012} only in the special case when $Z$ is a domain in $\bR^d$, 
$\cZ$ is the $\sigma$-algebra of the Borel subsets of $\bR^d$, $\mu$ is the Lebesgue measure 
on $\bR^d$ and $q=2\leq p_0\leq p_1$, but    
the same proof works also in our situation. 
For the convenience of the reader we present here the very nice argument from \cite{R2012} 
in our more general setting. The key role is played by the following lemma, 
which is an adaptation 
of Theorem 4 from \cite{R2012}. The notation $L_p(a,b)$ means the $L_p$ space of 
Borel-measurable real-functions on an interval $(a,b)$ with respect to the Lebesgue 
measure on $(a,b)$ for $-\infty\leq a<b\leq\infty$. 

\begin{lemma}                                                                 \label{lemma Riechwald}
Let $f\in\cL_1+\cL_{\infty}$ be a fixed function. 
Then there are bounded linear operators $S_1$ 
and $S_2$ mapping $\cL_1+\cL_{\infty}$ to $L_1(0,1)$ 
and $l_{\infty}$, respectively,  and 
there are also bounded linear operators $T_1$ and $T_2$ 
mapping $L_1(0,1)$ and $l_{\infty}$, 
respectively into $\cL_1+\cL_{\infty}$, such that 
\begin{equation}                                                               \label{decomposition}
f=T_1S_1f+T_2S_2f, 
\end{equation}
and for any $p\in[1,\infty]$ 
\begin{equation}                                                                                 \label{ST}
|S_1u|_{L_p(0,1)}\leq |u|_{\cL_p}, \quad |S_2u|_{l_p}\leq |u|_{\cL_p}, 
\quad |T_1v|_{\cL_p}\leq |v|_{L_p(0,1)},\quad |T_2w|_{\cL_p}\leq |w|_{l_p}
\end{equation}
for all $u\in \cL_p$, $v\in L_p(0,1)$ and $w\in l_p$. 
\end{lemma}
\begin{proof}
Though the proof of this lemma is just a repetition, 
in a more general setting, of that of Theorem 4 
from \cite{R2012}, for the convenience of the reader 
we present the full argument here. The main tool in the proof 
is a theorem of Calderon, Theorem 1 from \cite{C1966}, 
which under a stronger condition reads as follows.  
Let $\cL_p(Z_i)$ denote  the $L_p$-space of real functions 
on a $\sigma$-finite measure space 
$(Z_i,\cZ_i,\mu_i)$ for $i=1,2$, and let $f_i\in \cL_1(Z_i)+\cL_{\infty}(\cZ_i)$ 
such that 
$f_1^{\ast}(t)\geq f_2^{\ast}(t)$ for $t\in[0,\infty)$, where 
$$
f_i^{\ast}(t)=\inf\{\lambda\geq 0: \mu_i(|f_i|>\lambda)\leq t)\}, \quad t\geq0
$$
is the non-increasing right continuous rearrangement of $f_i$. 
Then there is a bounded linear operator $L$ from 
$\cL_1(Z_1)+\cL_{\infty}(\cZ_1)$ into 
$\cL_1(Z_2)+\cL_{\infty}(\cZ_2)$ such that $Lf_1=f_2$, 
\begin{equation}                                                                                 \label{Calderon}
|Lu|_{\cL_1(Z_2)}\leq |u|_{\cL_1(Z_1)} \quad\text{and}
\quad |Lv|_{\cL_{\infty}(Z_2)}\leq |v|_{\cL_{\infty}(Z_1)} 
\end{equation}
for $u\in \cL_1(Z_1)$ and $v\in \cL_{\infty}(Z_1)$. 

Since $f^{\ast}=(f^{\ast})^{\ast}$, one can apply 
Calderon's theorem to $f_1:=f$ and $f_2:=f^{\ast}$,  
to get an operator $L$ such that $Lf=f^{\ast}$ 
and \eqref{Calderon} holds with $(Z_1,\cZ_1,\mu_1):=(Z,\cZ,\mu)$ and 
$(Z_2,\cZ_2,\mu_2):=((0,\infty),\cB(0,\infty),dt)$. 
Define the operators $V_1$ and $V_2$ from 
$L_1(0,\infty)+L_{\infty}(0,\infty)$ to $L_1(0,1)$ and to $l_{\infty}$, 
respectively by 
$$
V_1u=u_{|(0,1)}, \text{i.e., the restriction of $u$ onto $(0,1)$, \,\,$V_2u
=\left(\int_{n-1}^nu(t)\,dt\right)_{n=1}^{\infty}$}. 
$$
Define also the operators 
$$
W_1: L_1(0,1)\to L_1(0,\infty)
\quad
\text{and}
\quad W_2:l_{\infty}\to L_{\infty}(0,\infty)
$$
by 
$$
W_1v(t):=
\begin{cases}v(t), &\text{if } t\in(0,1)\\
0&\text{if }{t\geq1}\,
\end{cases}
\quad\text{and}\quad W_2a (t):=\begin{cases}0, &\text{if } t\in(0,1)\\
a_n&\text{if }t\in[n,n+1),\,n=1,2,...\,.
\end{cases}
$$
Then for $g:=W_1V_1Lf+W_2V_2Lf\in L_1(0,\infty)+L_{\infty}(0,\infty)$ 
one has $g=f^{\ast}$ on $(0,1)$, and 
$$
g(t)=\int_{n-1}^nf^{\ast}_s\,ds\geq f^{\ast}(n)\geq f^{\ast}(t)
\quad\text{for}\,\,t\in[n,n+1)
\text{for integers $n\geq1$}. 
$$
Thus $f^{\ast}\leq g^{\ast}$, and one can apply Calderon's theorem 
again to get a bounded 
linear operator 
$H:L_1(0,\infty)+L_{\infty}(0,\infty)\to\cL_1+\cL_{\infty}$ 
such that $f=Hg$, and 
$H$ is a bounded operator from $L_1(0,\infty)$ to $\cL_1$ and from 
$L_{\infty}(0,\infty)$ to $\cL_{\infty}$, with operator norms not larger than 1. 
Hence it is easy to check 
that $S_i:=V_iL$ and $T_i:=HW_i$, $i=1,2$ satisfy \eqref{decomposition} and \eqref{ST} 
for $p=1,\infty$, and hence for all $p\in[1,\infty]$  by the Riesz-Thorin theorem. 
\end{proof}
\begin{proof}[Proof of Theorem \ref{theorem Lpq interpolation}] Consider first 
the case $1\leq q\leq p_0\leq p_1$.  
Notice that for fixed 
$q\geq1$ and for any $r\in[q,\infty]$
we have 
$$
\tilde L_r(0,1):=L_q(0,1)\cap L_r(0,1)=L_r(0,1)\quad \text{and}\quad   
\tilde l_r:=l_q\cap l_r=l_q. 
$$
Use also the notation $\tilde\cL_r:=\cL_q\cap\cL_r$.  
For an $f\in\tilde\cL_p=\cL_q\cap\cL_p\subset\cL_1+\cL_{\infty}$ 
let $S_i$ and $T_i$ 
denote the operators 
from the previous lemma.  Then clearly,  
$$
S_1:\tilde\cL_p\to\tilde L_p(0,1)=L_p(0,1)
=[\tilde L_{p_0}(0,1),\tilde L_{p_1}(0,1)]_{\theta}, 
\quad 
S_2:\tilde\cL_p\to \tilde l_p=l_q=[\tilde l_{p_0}, \tilde l_{p_1}]_{\theta}
$$
and by interpolation,
$$
\quad T_1: [\tilde L_{p_0}(0,1),\tilde L_{p_1}(0,1)]_{\theta}\to[\tilde \cL_{p_0},\tilde\cL_{p_1}]_{\theta}, \quad 
T_2:[\tilde l_{p_0}, \tilde l_{p_1}]_{\theta}\to 
[\tilde\cL_{p_0},\tilde\cL_{p_1}]_{\theta} 
$$ 
are bounded operators with operator norms not greater than 1. 
Hence taking $V:=[\tilde \cL_{p_0},\tilde \cL_{p_1}]_{\theta}$ norm 
in both sides of equation \eqref{decomposition} we get 
$$
|f|_{V}\leq |T_1S_1f|_{V}
+|T_2S_2f|_{V}\leq 2|f|_{\tilde \cL_p}. 
$$
Let now $f\in [\tilde\cL_{p_0},\tilde\cL_{p_1}]_{\theta}\subset \cL_1+\cL_{\infty}$, 
and denote again by $S_i$ and $T_i$ for $i=1$ the linear operators corresponding to 
$f$ by the above lemma.  
Then clearly, 
$$
T_1:\tilde L_p(0,1)\to\tilde\cL_p, \quad\text{and}
\quad  
T_2: \tilde l_p\to \tilde\cL_p, 
$$
and by interpolation 
$$
S_1:[\tilde\cL_{p_0},\tilde\cL_{p_1}]_{\theta}\to [\tilde L_{p_0}(0,1),\tilde L_{p_1}(0,1)]_{\theta}
=L_p(0,1)=\tilde L_p(0,1)
$$
and
$$
S_2:[\tilde\cL_{p_0},\tilde\cL_{p_1}]_{\theta}
\to [\tilde l_{p_0},\tilde l_{p_1}]_{\theta}=l_q=\tilde l_p 
$$
are bounded operators with operator norm not greater than 1. Hence 
$$
|f|_{\tilde\cL_p}\leq |T_1S_1f|_{\tilde\cL_p}
+|T_2S_2f|_{\tilde\cL_p}\leq 2|f|_{V},
$$
which finishes the proof of the theorem
when $1\leq q \leq p_0\leq p_1$. The theorem in the case $1\leq p_0\leq p_1\leq q$ 
can be proved in the same way with obvious changes. 
\end{proof}
\begin{theorem}                                  \label{theorem H}
Let the pair of separable Banach spaces $A_0$ and $A_1$ be an interpolation couple, 
and for $\theta\in[0,1]$ let $H^m_p(A_{\theta})$ denote the 
Bessel potential spaces of $A_{\theta}$-valued 
distributions for $m\in(-\infty,\infty)$ and $p\in(1,\infty)$. Then for $p_0,p_1\in(1,\infty)$ and   
$\theta\in(0,1)$ 
\begin{equation}                                                                  \label{H}
[H_{p_0}^m(A_0),H_{p_1}^m(A_1)]_{\theta}=H^m_{p_{\theta}}(A_{\theta})
\end{equation}
for any $m\in(-\infty,\infty)$, where $p_{\theta}=((1-\theta)/p_0+\theta/p_1)^{-1}$. 
\end{theorem}
\begin{proof}
We have equation \eqref{H} for $m=0$ by (ii) in 
Theorem \ref{theorem i1}. 
Since 
$$
(1-\Delta)^{-m/2}: H^0_{r}(A_{\vartheta})\to H^m_r(A_{\vartheta})
$$
is a bounded operator with operator norm 1 
for every $r\in(1,\infty)$ and $\vartheta\in[0,1]$, by 
(i) in Theorem \ref{theorem i1} the operator 
$(1-\Delta)^{-m/2}:
H^0_{p_\theta}(A_{\theta})
\to 
[H_{p_0}^m(A_0),H_{p_1}^m(A_1)]_{\theta}=:H$ 
is bounded with operator norm not greater than 1. 
Hence, taking into account that 
$(1-\Delta)^{-m/2}$ maps 
$H^0_{p_\theta}(A_{\theta})$ 
onto 
$H^m_{p_{\theta}}(A_{\theta})$, 
we have $H^m_{p_{\theta}}(A_{\theta})\subset H$. 
By (i) in Theorem \ref{theorem i1} 
we also have that $(1-\Delta)^{m/2}:H\to H^0_{p_{\theta}}(A_{\theta})$ 
is a bounded operator with operator norm not greater than 1. 
Hence, taking into account that  
$(1-\Delta)^{m/2}:H^m_{p_{\theta}}(A_{\theta})\to H^0_{p_{\theta}}(A_{\theta})$ 
is a bounded operator 
with operator norm 1, for $w\in H^m_{p_{\theta}}(A_{\theta})\subset H$ 
we get 
$$
|w|_{H}=|(1-\Delta^{-m/2})(1-\Delta^{m/2})w|_H
\leq |(1-\Delta^{m/2})w|_{H^0_{p_{\theta}}}(A_{\theta})
=|w|_{H^m_{p_\theta}(A_{\theta})}, 
$$ 
and 
$$
|w|_{H^m_{p_{\theta}}(A_{\theta})}
=|(1-\Delta^{-m/2})(1-\Delta^{m/2})w|_{H^m_{p_{\theta}}(A_{\theta})}
\leq|(1-\Delta^{m/2})w|_{H^0_{p_\theta}(A_{\theta})}
\leq |w|_{H}, 
$$
which finishes the proof of the lemma. 
\end{proof}
Hence by virtue of Theorem \ref{theorem Lpq interpolation} we have the following 
corollary. 
\begin{corollary}                                   \label{corollary H}
For $1\leq p_0\leq p_1$, $q\notin(p_0,p_1)$ and   
$\theta\in(0,1)$ 
\begin{equation*}                                                         
[H_{p_0}^m(\cL_{p_0,q}),H_{p_1}^m(\cL_{p_1,q})]_{\theta}
=H^m_{p_{\theta}}(\cL_{p_{\theta},q})
\end{equation*}
with equivalent norms for any $m\in(-\infty,\infty)$, 
where $p_{\theta}=((1-\theta)/p_0+\theta/p_1)^{-1}$ and 
$\cL_{r,q}=\cL_r\cap\cL_q$ for any $r,q\in[1,\infty)$.
\end{corollary}
\mysection{$L_p$ estimates}
Let Assumptions \ref{assumption L} 
and \ref{assumption eta} hold, and  
let $u=(u_t)_{t\in[0,T]}$ be a $W^{n+2}_p$-valued 
solution to \eqref{eq1}-\eqref{ini1} for some integer $n\geq 0$. 
Then by an application of Lemma \ref{lemma Ito1} 
we have 
\begin{align}
|D^nu_t|_{L_p}^p = &  |D^n\psi|_{L_p}^p
+p\int_0^tQ_{n,p}(s,u_s,f_s,g_s)+{Q}^{\xi}_{n,p}(s,u_s)
+{Q}^{\eta}_{n,p}(s,u_s)\,ds                                                                           \nonumber\\
&
+p\int_0^tD_\alpha(\cM^r_su_s+g_s^r)\,dx\,dw_s^r                                      \nonumber\\
&
+\int_0^t\int_Z\int_{\bR^d}\{( \sum_{|\alpha|=n}|D_\alpha u_{s-}
+D_\alpha(I^{\eta}u_{s-}+h_s)|^2 )^{p/2}-|D^nu_{s-}|^p\nonumber\\
&
-p\sum_{|\alpha|=n}|D^nu_{s-}|^{p-2}D_\alpha u_{s-} 
D_\alpha(I^{\eta}u_{s-}+h_s)\}\,dx\,\pi(dz,ds)                                                  \label{L_p}
\end{align}
holds almost surely for all $t\in[0,T]$, where 
\begin{align}
Q_{n,p}(t,v,f,g)= & \int_{\bR^d}p|D^nv|^{p-2}
\{\sum_{|\alpha|=n}v_\alpha D_\alpha(\cL v+f)
+\frac{1}{2}\sum_{r=1}^\infty\sum_{|\alpha|=n}
|D_\alpha (\cM^rv+g^r)|^2\}\,dx                                                       \nonumber\\
&
+\int_{\bR^d}\frac{1}{2}p(p-2)|D^nv|^{p-4}\sum_{r=1}^\infty|\sum_{|\alpha|=n}
v_\alpha D_\alpha(\cM^rv+g^r)|^2\,dx,                                                     \label{drift L}
\end{align}
\begin{equation}                                                                          \label{Qxi}
{Q}^{\xi}_{n,p}(t,v)=\int_Z\int_{\bR^d}
p|D^nv|^{p-2}\sum_{|\alpha|=n}v_{\alpha}(J^{\xi} v)_{\alpha}\,dx\,\nu(dz)
\end{equation}
\begin{equation}                                                                          \label{Qeta}
{Q}^{\eta}_{n,p}(t,v)=\int_Z\int_{\bR^d}
p|D^nv|^{p-2}\sum_{|\alpha|=n}v_{\alpha}(J^{\eta} v)_{\alpha}\,dx\,\mu(dz)
\end{equation}
for $v\in W^{n+2}_2$,  for each $f\in W^m_p$, 
$g\in W^{m+1}_p(l_2)$, $h\in W_p^{m+1}(\cL_{p,2})$, 
$\omega\in\Omega$ and $t\in[0,T]$.   
Recall that the notation $v_{\alpha}=D_{\alpha}v$ is often used. 
In order to estimate the right-hand side of \eqref{L_p}, 
we also define for integers $n\in[0,m]$ and $p\geq2$ the 
``$p$-form"  
$$                 
\hat Q_{n,p}(t,v,h)=\int_Z\int_{\bR^d}p|D^nv|^{p-2}
\sum_{|\alpha|=n}v_{\alpha}\big(J^{\eta}v)_{\alpha}
\,dx\,\mu(dz)
$$
\begin{equation}                                                                                \label{drift}   
+\int_Z\int_{\bR^d}|D^n(T^{\eta}v+h)|^{p}-|D^nv|^p
-p|D^nv|^{p-2}
\sum_{|\alpha|=n}
v_{\alpha}\big((I^{\eta}v)_{\alpha}+h_{\alpha}\big) \,dx\,\mu(dz)
\end{equation}
for $v\in W^{n+2}_2$,  $h\in W_p^{m+1}(\cL_{p,2})$, 
for each $\omega\in\Omega$ and $t\in[0,T]$.   

\begin{proposition}                                                                                               \label{proposition L}
Let Assumption \ref{assumption L} hold. 
Then for integers $n\in[0,m]$ and any 
$p\in[2,\infty)$ there is a 
constant $N=N(d,p,m,K)$ such that 
\begin{equation}                                                                                                  \label{estimate L}
Q_{n,p}(v,t,f,g)
\leq N\big(|v|^p_{W^n_p}
+|f|^p_{W^n_p}+|g|^p_{W^{n+1}_p(l_2)}\big)
\end{equation}
for all $v\in W^{n+2}_p$, $f\in W^n_p$, $g\in W^{n+1}_p(l_2)$, 
$\omega\in\Omega$ and $t\in[0,T]$.  
\end{proposition}

\begin{proof} This estimate is proved in \cite{GGK} in a more general setting. 
\end{proof}

\begin{proposition}                                                                         \label{proposition Qxieta}
Let Assumption \ref{assumption xi} 
hold. Then the following statements hold:
\begin{enumerate}[(i)]
\item If $p=2^k$ for an integer $k\geq1$ then for integers $n\in[0,m]$ we have 
\begin{equation}                                       \label{estimate tilde}
{Q}^{\xi}_{n,p}(t,v)\leq N|v|^p_{W^n_p}
\end{equation}
 for $v\in W^{n+2}_p$, $\omega\in\Omega$ and $t\in[0,T]$,  
with a constant $N=N(d,p,m,K,K_{\xi})$.
\item For integers $n\in[0,m]$ and for all $p\geq2$ 
\begin{equation}                                                                                \label{Jpestimate}
\int_Z\int_{\bR^d}|D^nv|^{p-2}
\sum_{|\alpha|=n}v_{\alpha}J^{\xi}v_{\alpha}\,dx\,\nu(dz)
\leq N|v|^p_{W^n_p}
\end{equation}
for $v\in W^{n+2}_p$, $\omega\in\Omega$ and $t\in[0,T]$,  
with a constant $N=N(d,p,m,K,K_{\xi})$. 
\end{enumerate}  
\end{proposition}

\begin{proof}
Statement (i) is proved in \cite{DGW}, see Theorem 4.1 therein. To prove (ii) 
notice that 
\begin{equation*}
p|D^nv|^{p-2}v_{\alpha}J^{\xi}v_{\alpha}=
p|D^nv|^{p-2}v_{\alpha}(I^{\xi}v_{\alpha}-v_{\alpha i}\xi^i)
\end{equation*}
$$
=p|D^nv|^{p-2}v_{\alpha}I^{\xi}v_{\alpha}+J^{\xi}|D^nv|^p-I^{\xi}|D^nv|^p. 
$$
Consider the vector $D^nv=(v_{\alpha})_{|\alpha|=n}$. 
Then $|D^nv|^p$ is a convex function 
of $D^nv$. Consequently, 
$$
I^{\xi}|D^nv|^p-p|D^nv|^{p-2}v_{\alpha}I^{\xi}v_{\alpha}\geq0. 
$$
Hence, taking into account Lemma \ref{lemma intIJ} (ii), with $\xi$ 
in place of $\eta$,  
we get \eqref{Jpestimate}.
\end{proof}

\begin{proposition}                                                                 \label{proposition hat}
Let Assumption \ref{assumption eta} hold and let $n\in[0,m]$ be an integer. 
Then 
for $p=2$ there is a constant 
$N=N(d,m, K,K_\eta)$ such that for all $\omega\in\Omega$ 
and  $t\in[0,T]$  
\begin{equation}                                                                                               \label{estimate hat}
\hat Q_{n,2}(v,t,h)\leq N\big(|v|^2_{W^n_2}+|h|^2_{W^{n+1}_2(\cL_2)}\big)
\end{equation}
for all $v\in W^{n+2}_p$ and $h\in W^{n+1}_2(\cL_2)$.   
Moreover, for $p\in[4,\infty)$ we have a constant 
$N=N(K,d,m,p,K_\eta)$ such that for 
$\omega\in\Omega$ and  $t\in[0,T]$   
\begin{equation}                                                                                               \label{estimate hat1}
\hat Q_{n,p}(v,t,h)\leq N(|v|^p_{W^n_p}+|h|^p_{W^{n+2}_p(\cL_{p,2})})
\end{equation}
 for all $v\in W^{n+2}_p$ and $h\in W^{n+2}_p(\cL_{p,2})$. 
For $h=0$ this estimate holds for all $p\in[2,\infty)$. 
\end{proposition}

To prove this proposition we recall the notation 
$T^{\eta}$ for the operator defined 
by $T^{\eta}\varphi(x)=v(x+\eta(x))$ for functions 
$\varphi$ on $\bR^d$, and for multi-numbers 
$\alpha\neq\epsilon$ 
introduce the notations  
\begin{equation}                                                          \label{Ga}
G^{(\alpha)}(v):=
\sum_{1\leq k\leq|\alpha|}
\sum_{\kappa_0\sqcup\kappa_1\sqcup...\sqcup\kappa_k=s_{\alpha}}
\eta^{i_1}_{\alpha(\kappa_1)}\cdots\eta^{i_k}_{\alpha(\kappa_k)}
T^{\eta}v_{\alpha(\kappa_0)i_1i_2...i_k}
\end{equation}
for $v\in W^m_p$, where $s_{\alpha}:=\{1,2,...,|\alpha|\}$, and   
summation over $\kappa_0\sqcup\kappa_1\sqcup...\sqcup\kappa_k=s_{\alpha}$ 
means summation over all partitions of $s_{\alpha}$ into disjoint subsets 
$\kappa_0$,...,$\kappa_k$ such that $\kappa_j\neq\emptyset$  
for $j\geq1$, and two partitions 
$\kappa_0\sqcup\kappa_1\sqcup...\sqcup\kappa_k=s_{\alpha}$ 
and 
$\kappa_0'\sqcup\kappa_1'\sqcup...\sqcup\kappa_k'=s_{\alpha}$ 
are different if either 
$\kappa_0\neq\kappa_0'$ or for some $i=1,2,...,k$ the set $\kappa_i$ 
is different from each of the sets $\kappa_j'$ for $j=1,2,...k$. For a 
subset $\kappa\subset s_{\alpha}$ the notation $\alpha(\kappa)$ 
means the multi-number 
$\alpha_{j_1}...\alpha_{j_r}$, where $j_1$,...,$j_r$ the elements of $\kappa$ 
in increasing order. When $\kappa_0$ is the empty set, then 
$\alpha(\kappa_0)=\epsilon$, the multi-number of length zero. 
Recall that $v_{\alpha}=D_{\alpha}v$ and $D_{\epsilon}v=v$. 

Noticing that for $i=1,2,...,k$ 
$$
(T^{\eta}v)_i-T^{\eta}v_i=\eta^k_iT^{\eta}v_k,  
$$
by induction on the length of multi-numbers $\alpha$ we get 
\begin{equation}                                                                          \label{TG}
(T^{\eta}v)_{\alpha}-T^{\eta}v_{\alpha}
=G^{(\alpha)}\quad\text{for $|\alpha|\geq0$}.
\end{equation}

We will prove Proposition \ref{proposition hat} by the help of the following 
lemmas. 
\begin{lemma}                             \label{lemma estimate}
Let Assumption \ref{assumption eta} hold. 
Then for $p\geq2$  and integers $n\in[0,m]$ we have 
$$
p|D^nv|^{p-2}\sum_{|\alpha|=n}v_{\alpha}((J^{\eta}v)_{\alpha}-(I^{\eta}v)_{\alpha})=
J^{\eta}|D^nv|^{p}-I^{\eta}|D^nv|^{p}
$$
\begin{equation*}
-p|D^nv|^{p-2}
\sum_{1\leq|\alpha|=n}
\sum_{\kappa_0\sqcup\kappa_1=s_{\alpha}}
v_{\alpha}\eta^k_{\alpha(\kappa_1)}v_{\alpha(\kappa_0)k}
\end{equation*}
for all $v\in W^{n+1}_p$, $\omega\in\Omega$, $z\in Z$ and $t\in[0,T]$. 
\end{lemma}
\begin{proof}
Clearly,
$$
p|D^nv|^{p-2}\sum_{|\alpha|=n}v_{\alpha}\big((J^{\eta}v)_{\alpha}-(I^{\eta}v)_{\alpha}\big)
=-p|D^n|^{p-2}\sum_{|\alpha|=n}v_{\alpha}(v_i\eta^i)_{\alpha}
$$
$$
=-p|D^nv|^{p-2}\sum_{|\alpha|=n}v_{\alpha}v_{\alpha i}\eta^i 
-p|D^n|^{p-2}\sum_{1\leq|\alpha|=n}v_{\alpha}((v_i\eta^i)_{\alpha}-v_{\alpha i}\eta^i) 
$$
$$
=\eta^iD_i|D^nv|^{p}
-p|D^nv|^{p-2}\sum_{1\leq|\alpha|=n}
\sum_{\kappa_0\sqcup\kappa_1=s_{\alpha}}
v_{\alpha}\eta^k_{\alpha(\kappa_1)}v_{\alpha(\kappa_0)k}
$$
\begin{equation*}                                                           
=J^{\eta}|D^nv|^p-I^{\eta}|D^nv|^p
-p|D^nv|^{p-2}\sum_{1\leq|\alpha|=n}
\sum_{\kappa_0\sqcup\kappa_1=s_{\alpha}}
v_{\alpha}\eta^k_{\alpha(\kappa_1)}v_{\alpha(\kappa_0)k}. 
\end{equation*}
\end{proof}
For the next lemmas consider for integers $n\geq0$ the expressions 
\begin{equation*}                                                    
|G_{n}(v)|^p:=\big(\sum_{|\alpha|=n}|G^{(\alpha)}(v)|^2\big)^{p/2}, 
\end{equation*} 
\begin{equation*}                                                
B_{n,p}(v,h):=|D^{n}(T^{\eta}v+h)|^{p}-|D^nv|^p
-p|D^nv|^{p-2}\sum_{|\alpha|=n}v_\alpha h_\alpha
\end{equation*} 
\begin{equation*}                                                       
H_{n,p}(v,h):=\int_{Z}\Big|\int_{\bR^d}
I^{\eta}(|D^nv(x)|^{p-2}v_{\alpha}(x))h_{\alpha}(x,z)\,dx\Big|\,\mu(dz)
\end{equation*}
and 
\begin{equation*}                                                       
K_{n,p}(v,h):=\int_{Z}\int_{\bR^d}
T^{\eta}|D^nv(x)|^{p-2}|D^nh(x,z)|^2\,dx\,\mu(dz)
\end{equation*}
for $v\in W^{n+1}_p$ and $h\in W^{n+1}_p(\cL_{p,2})$, 
where $G^{(\alpha)}$ is defined in \eqref{Ga}. 
\begin{lemma}                                                        \label{lemma B}                                                         
Let Assumption \ref{assumption eta} hold with $m\geq0$. 
Then for 
\begin{equation*}
\bar B_{n,p}(v,h):=B_{n,p}(v,h)-I^{\eta}|D^{n}v|^p
-pT^{\eta}|D^nv|^{p-2}T^{\eta}v_{\alpha}G^{(\alpha)}(v)
-pI^{\eta}(|D^nv|^{p-2}v_\alpha) h_\alpha 
\end{equation*}
for any $p\geq2$  and integers $n\in[0,m]$ we have 
$$
|\bar B_{n,p}(v,h)|
\leq
NT^{\eta}|D^nv|^{p-2}|G_n(v)|^{2}
+NT^{\eta}|D^nv|^{p-2}|D^nh|^2+N|G_n(v)|^p+N|D^nh|^p
$$
for $v\in W^{n+1}_p$ and $h\in W^{n+1}_p(\cL_{p,2})$, 
with a constant $N=N(p,d,m)$.  
\end{lemma}
\begin{proof}
Let ${\bf x}=({\bf x}_{\alpha})$ and ${\bf y}=({\bf y}_{\alpha})$ 
denote the vectors with coordinates 
$$
{\bf x}_{\alpha}:=T^{\eta}v_{\alpha},  
\quad
{\bf y}_{\alpha}:=G^{(\alpha)}+h_{\alpha} 
$$
for multi-numbers $\alpha$ of length $n$. Then taking into account \eqref{TG} we have 
$$
|{\bf x}+{\bf y}|^p=|D^{n}(T^{\eta}v+h)|^{p},\quad |{\bf x}|^p=T^{\eta}|D^nv|^p, 
$$
and 
by Taylor's formula there is a constant $N=N(d,p,n)$ such that 
$$
0\leq|{\bf x}+{\bf y}|^p-|{\bf x}|^p-p|{\bf x}|^{p-2}{\bf x}_{\alpha}{\bf y}_{\alpha}\leq 
N(|{\bf x}|^{p-2}|{\bf y}|^2+|{\bf y}|^p). 
$$
Hence writing $G$ for the vector $(G^{(\alpha)})_{|\alpha|=n}$  and noticing 
$$ 
T^{\eta}|D^nv|^{p-2}T^{\eta}v_{\alpha}-|D^nv|^{p-2}v_{\alpha}
=I^{\eta}(|D^nv|^{p-2}v_{\alpha}), 
$$
we get 
\begin{equation}                                                     \label{BB}
\bar B_{n,p}(v,h)
\leq
NT^{\eta}|D^nv|^{p-2}|G|^{2}
+NT^{\eta}|D^nv|^{p-2}|D^nh|^2+N|G|^p+N|D^nh|^p
\end{equation}
with a constant $N=N(d,n,p)$. 
\end{proof}
\begin{lemma}                                                           \label{lemma HK}
Let Assumption \ref{assumption eta} hold. Then for all $\omega\in\Omega$, 
$t\in[0,T]$ and integers $n\in[0,m]$ we have  
\begin{equation}                                                        \label{HKestimate2}
H_{n,2}(v,h)\leq 
N(|v|^{2}_{W^n_2}+|h|^2_{W^{n+1}_2(\cL_{2})}),\quad 
K_{n,2}(v,h)
\leq |h|^2_{W^{n}_2(\cL_2)}
\end{equation}
for $v\in W^{n+2}_2$ and $h\in W^{n+1}_2(\cL_2)$ 
with a constant $N=N(d,m,K,K_{\eta})$. For $p\in[4,\infty)$ and integers 
$n\in[0,m]$ 
\begin{equation}                                                        \label{HKestimate}
H_{n,p}(v,h)
\leq N(|v|^{p}_{W^n_p}+|h|^p_{W^{n+2}_p(\cL_{p,2})}), 
\quad 
K_{n,p}(v,h)
\leq N(|v|^{p}_{W^n_p}+|h|^p_{W^{n+2}_p(\cL_p)})
\end{equation}
for all $v\in W^{n+2}_p$, $h\in W^{n+2}_p(\cL_{p,2})$, $\omega\in\Omega$ 
and $t\in[0,T]$ with a constant $N=N(p,d,m,K,K_{\eta})$. 
\end{lemma}
\begin{proof}
The second estimate in \eqref{HKestimate2} is obvious.  
By Taylor's formula, Fubini's theorem and by change of variables we have 
$$
H_{n,2}(h,v)=\int_{Z}\Big|\int_{\bR^d}\int_0^1
T^{\theta\eta}v_{\alpha i}(x)\eta^i_{t,z}(x)h_{\alpha}(x,z)\,d\theta\,dx\Big|\,\mu(dz)
$$
$$
\leq \int_0^{1}\int_{Z}\Big|\int_{\bR^d}
v_{\alpha i}(x)\varrho^{\alpha i}_{t,\theta,z}(x)\,dx\Big|\,\mu(dz)\,d\theta
$$
with 
$$
\varrho^{\alpha i}_{t,\theta,z}(x)
=h_{\alpha}(\tau^{-1}_{t,\theta,z}(x),z)
\eta^i_{t,z}(\tau^{-1}_{t,\theta,z}(x),z)|D\tau^{-1}_{t,\theta,z}(x)|, 
$$ 
where 
$\tau^{-1}_{t,\theta,z}(\cdot)$ is the inverse of the diffeomorphism 
$\tau_{t,\theta,z}(x)=x+\theta\eta_{t,z}(x)$. Hence by integration by parts we obtain 
$$
H_{n,2}(h,v)\leq \int_0^{1}\int_{Z}\Big|\int_{\bR^d}
v_{\alpha}(x)D_i\varrho^{\alpha i}_{t,\theta,z}(x)\,dx\Big|\,\mu(dz)\,d\theta. 
$$
Due to Assumption \ref{assumption eta}, Lemma \ref{lemma diff}, 
the Cauchy-Schwarz and H\"older's inequalities   
there is a constant $N=N(d,K,m)$ such that
$$
|D_i\varrho^{\alpha i}_{t,\theta,z}(x)|
\leq N \bar{\eta}\big(|D^nh|(\tau^{-1}_{t,z,\theta}(x),z)
+|D^{n+1}h|(\tau^{-1}_{t,z,\theta}(x),z)\big).
$$
Hence by H\"older's inequality 
$$
H_{n,2}(v,h) \leq N|\bar\eta D^nv|_{\cL_2(L_2)}
|\sum_{k=n}^{n+1}|D^{k}h|_{\cL_2(L_2)}
\leq NK_{\eta}|D^nv|_{L_2}|h|_{W^{n+1}_2(\cL_2)}
$$
with $N=N(K,d,m)$, which proves the first inequality in \eqref{HKestimate2}. 
Let $p\geq4$. Then 
$$
H_{n,p}(v,h)\leq \sum_{j=0}^1H^{(j)}_p(v,h)\quad\text{and}
\quad
K_{n,p}(v,h)\leq\sum_{j=0}^2K^{(j)}_p(v,h) 
$$
with
$$
H^{(0)}_p(v,h)
=\int_{Z}\Big|\int_{\bR^d}
J^{\eta}(|D^nv(x)|^{p-2}v_{\alpha})h_{\alpha}(x,z)\,dx\Big|\,\mu(dz), 
$$
$$
H^{(1)}_p(v,h)
=\int_{Z}\Big|\int_{\bR^d}
D_i(|D^nv(x)|^{p-2}v_{\alpha})\eta^ih_{\alpha}(x,z)\,dx\Big|\,\mu(dz) 
$$
and
$$
K^{(0)}_p(v,h)
=\int_{Z}\Big|\int_{\bR^d}
J^{\eta}|D^nv(x)|^{p-2}|D^nh(x,z)|^2\,dx\Big|\,\mu(dz), 
$$
$$
K^{(1)}_p(v,h)
=\int_{Z}\Big|\int_{\bR^d}
D_i|D^nv(x)|^{p-2}\eta^i|D^nh(x,z)|^2\,dx\Big|\,\mu(dz),  
$$
$$
K^{(2)}_p(v,h)
=\int_{Z}\int_{\bR^d}
|D^nv(x)|^{p-2}|D^nh(x,z)|^2\,dx\,\mu(dz).  
$$
By Fubini's theorem and H\"older's inequality 
$$ 
K^{(2)}_p(v,h)
\leq \int_{\bR^d}|D^nv(x)|^{p-2}||D^nh(x,\cdot)|^2_{\cL_2}\,dx
\leq |v|^{p-2}_{W^n_p}|h|^2_{W^n_p(\cL_2)}. 
$$
Since $p\geq4$, by Taylor's formula, Fubini's theorem, change 
of variables, integration by parts and using Assumption \ref{assumption eta} 
we get 
$$
H_p^{(0)}(v,h)\leq 
\int_0^1\int_{Z}\Big|\int_{\bR^d}(1-\theta)T^{\theta\eta}
(|D^nv(x)|^{p-2}v_{\alpha}(x))_{ij}\eta^i\eta^jh_{\alpha}\,dx\Big|\,\mu(dz)\,d\theta
$$
$$
\leq N \int_0^1\int_{Z}\int_{\bR^d}\bar\eta^2(z)
|D^nv|^{p-1}\sum_{k=n}^{n+2}|D^kh|(\tau^{-1}(x),z)\,dx\,\mu(dz)\,d\theta
$$
with $N=N(d,p,m,K)$. Hence by H\"older's inequality, change of variables, Fubini's theorem 
and using $|\bar\eta|\leq K$ we obtain 
\begin{equation}                                                               \label{H0}
H_p^{(0)}(v,h)\leq N
|\bar\eta^{2/p}|D^nv||^{p-1}_{\cL_p(L_p)}
\sum_{k=n}^{n+2}|D^kh|_{\cL_p(L_p)}
\leq NK_{\eta}^{2(p-1)/p}|D^nv|^{p-1}_{L_p}|h|_{W^{n+2}_p(\cL_p)}. 
\end{equation}
Similarly, with a constant $N=N(K,d,m,p,K_{\eta})$ we have
\begin{equation}                                                         \label{K0}
K_p^{(0)}(v,h)\leq N|D^nv|^{p-2}_{\cL_p(L_p)}
|h|^2_{W^{n+2}_p(\cL_p)}. 
\end{equation}
By integration by parts, using Assumption \ref{assumption eta}, 
Cauchy-Schwarz and H\"older inequalities 
we get 
$$
H_p^{(1)}(v,h)=\int_{Z}\Big|\int_{\bR^d}
|D^nv(x)|^{p-2}v_{\alpha}D_i(\eta^ih_{\alpha}(x,z))\,dx\Big|\,\mu(dz) 
$$
$$
\leq \int_{Z}\int_{\bR^d}
|D^nv(x)|^{p-1}\bar\eta(z)\sum_{k=n}^{n+1}|D^kh(x,z)|\,dx\,\mu(dz) 
$$
\begin{equation}                                                                   \label{K1}
\leq K_{\eta}\int_{\bR^d}
|D^nv(x)|^{p-1}\sum_{k=n}^{n+1}|D^kh(x,\cdot)|_{\cL_2}\,dx\,\mu(dz)
\leq  K_{\eta}|v|^{p-1}_{W^n_p}|h|_{W_p^{n+1}(\cL_2)}
\end{equation}
Similarly we have 
$$
K_p^{(1)}(v,h)\leq |v|^{p-2}_{W^n_p}|h|_{W^{n+1}_p(\cL_2)}. 
$$
Combining this with \eqref{H0} through \eqref{K1} and using Young's inequality 
we get \eqref{HKestimate}. 
\end{proof}
\begin{proof}[Proof of Proposition \ref{proposition hat}]
Set
$$ 
A:=p|D^nv|^{p-2}
\sum_{|\alpha|=n}v_{\alpha}\big((J^{\eta}v)_{\alpha}-(I^{\eta}v)_{\alpha}\big).  
$$
Then by Lemma \ref{lemma estimate}
$$
A=J^{\eta}|D^nv|^p-I^{\eta}|D^nv|^p
$$ 
\begin{equation}                                                 \label{FG}
-p|D^nv|^{p-2}\sum_{1\leq|\alpha|=n}
\sum_{\kappa_0\sqcup\kappa_1=s_{\alpha}}\eta^i_{\alpha(\kappa_1)}
v_{\alpha}v_{\alpha(\kappa_0)i}. 
\end{equation}
By Lemma \ref{lemma B} for 
$$
B=|D^{n}(T^{\eta}v+h)|^{p}-|D^nv|^p
-p|D^nv|^{p-2}\sum_{|\alpha|=n}v_\alpha h_\alpha, 
$$
we have
$$
B\leq I^{\eta}|D^{n}v|^p
+pT^{\eta}|D^nv|^{p-2}T^{\eta}v_{\alpha}G^{(\alpha)}
+pI^{\eta}(|D^nv|^{p-2}v_\alpha) h_\alpha
$$
$$
+NT^{\eta}|D^nv|^{p-2}|G|^{2}+NT^{\eta}|D^nv|^{p-2}|D^nh|^2+N|G|^p+N|D^nh|^p
$$
with a constant $N=N(d,n,p)$. Thus introducing the notations 
$$
\tilde G^{(\alpha)}:=G^{(\alpha)}-\sum_{\kappa_0\sqcup\kappa_1=s_{\alpha}}
\eta^{i}_{\alpha(\kappa_1)}
T^{\eta}v_{\alpha(\kappa_0)i}
$$
and
$$
R:=T^{\eta}|D^nv|^{p-2}|G|^{2}+|G|^p+|D^nh|^p
+T^{\eta}|D^nv|^{p-2}T^{\eta}v_{\alpha}\tilde G^{(\alpha)}, 
$$
we have 
$$
B\leq I^{\eta}|D^{n}v|^p+pT^{\eta}|D^nv|^{p-2}\sum_{|\alpha|=n}
\sum_{\kappa_0\sqcup\kappa_1=s_{\alpha}}
\eta^{i}_{\alpha(\kappa_1)}
T^{\eta}v_{\alpha}T^{\eta}v_{\alpha(\kappa_0)i}
$$
$$
+pI^{\eta}(|D^nv|^{p-2}v_{\alpha})h_{\alpha}
+NT^{\eta}|D^nv|^{p-2}|D^nh|^2+NR  
$$
with a constant $N=N(d,n,p)$. 
Combining this with equation \eqref{FG} and noticing  that 
$$
T^{\eta}|D^nv|^{p-2}
T^{\eta}v_{\alpha}
T^{\eta}v_{\alpha(\kappa_0)i}-|D^nv|^{p-2}
v_{\alpha}v_{\alpha(\kappa_0)i}
=I^{\eta}(|D^nv|^{p-2}v_{\alpha}v_{\alpha(\kappa_0)i}) 
$$
we obtain 
$$
A+B=J^{\eta}(|D^nv|^p)
+p\sum_{|\alpha|=n}
\sum_{\kappa_0\sqcup\kappa_1=s_{\alpha}}\eta^i_{\alpha(\kappa_1)}
I^{\eta}(|D^nv|^{p-2}v_{\alpha}v_{\alpha(\kappa_0)i})
$$
\begin{equation}                                                                                            \label{A+B}
+p\sum_{|\alpha|=n}h_{\alpha}I^{\eta}(|D^nv|^{p-2}v_{\alpha})
+NT^{\eta}|D^nv|^{p-2}|D^nh|^2+NR.
\end{equation}
By Lemma \ref{lemma intIJ} (i) and (ii) 
we have a constant $N=N(K,d)$ such that 
\begin{equation}                                                                                              \label{0}
\int_Z\int_{\bR^d}J(|D^nv|^p)\,dx\,\mu(dz)
\leq N|D^nv|_{L_p}^p\int_Z\bar\eta^2\,\mu(dz), 
\end{equation}
and
\begin{equation}                                                                                              \label{0.5}
\int_Z\int_{\bR^d}\eta^i_{\alpha(\kappa_1)}
I(|D^nv|^{p-2}v_{\alpha}v_{\alpha(\kappa_0)i})\,dx
\leq N|D^nv|_{L_P}^p\int_Z\bar\eta^2\,\mu(dz). 
\end{equation}
Due to Assumption \ref{assumption eta}, taking into account that $p\geq2$ and 
using Young's inequality we get 
$$
|R|\leq N\bar\eta^2\sum_{1\leq k\leq n}T^{\eta}|D^kv|^{p}+|D^nh|^p,  
$$
$$                                                                                            
\int_{Z}\int_{\bR^d}R\,dx\,\mu(dz)
\leq N\int_{Z}\bar\eta^2(z)\,\mu(dz)|v|^p_{W^n_p}+|D^nh|^p_{L_p(\cL_p)}  
$$
with a constant $N=N(d,n,p)$.  Combining this with 
estimates \eqref{A+B} through \eqref{0.5} and using Lemma \ref{lemma HK} 
we finish the proof of the proposition. 
\end{proof}
Introduce also the expressions 
\begin{equation*}                                                                                         \label{DM1}
{\frak P}^2_{n,p}(t,v,g)
:=\sum_{r=1}^{\infty}(p|D^nv|^{p-2}D_\alpha v,D_\alpha(\cM^rv+g^r))^2
\end{equation*}
\begin{equation*}                                                                                           \label{DM2}
{\frak R}^2_{n,p}(t,v,h)
:=\int_Z
\big(p|D^nv|^{p-2}v_\alpha, (I^{\eta}v)_{\alpha}+h_{\alpha}\big)^2\,\mu(dz)
\end{equation*}
\begin{align}                                                                                           \label{DM3}
{\frak Q}_{n,p}(t,v,h)
:=\int_Z\int_{\bR^d}&\{|D^n(v+I^{\eta}v+h)|^{p}
-|D^nv|^p    \nonumber\\
&
-p|D^nv|^{p-2}D_\alpha v
D_\alpha(I^{\eta}v+h)\}\,dx\,\mu(dz)  
\end{align}
for $v\in W^{m+1}_p$, $g\in W^{m+1}_p(l_2)$, $h\in W^{m+1}_p(\cL_{p,2})$, 
$\omega\in\Omega$ and $t\in[0,T]$, where repeated indices $\alpha$ 
mean summation over all multi-numbers of length $n$. 
\begin{proposition}                                                               \label{proposition DM}
Let $n\in[0,m]$ be 
an integer and $p\in[2,\infty)$.  Then the following estimates hold
for all $(\omega,t)\in\Omega\times[0,T]$.
\begin{enumerate}[(i)]
\item
If Assumption \ref{assumption L} is satisfied then 
\begin{equation}                                                                   \label{DM1estimate}
{\frak P}^2_{n,p}(t,v,g)
\leq N(|v|^{2p}_{W^n_p}+|v|^{2p-2}_{W^n_p}|g|^{2}_{W^{n}_p})
\end{equation}
for all $v\in W^{n+1}_p$ and $g\in W^{n+1}$, 
with a constant $N=N(d,m,p,K)$. 
\item
If Assumption \ref{assumption eta} is satisfied 
then 
\begin{equation}                                                                   \label{DM3estimate}
{\frak Q}_{n,p}(t,v,h)
\leq N(|v|^{p}_{W^{n+1}_p}
+|h|^{p}_{W^n_{p,2}})
\end{equation}
for all $v\in W^{n+1}_p$ and $h\in W^{n}_p(\cL_{p,2})$ 
with $N=N(d,m,p,K,K_{\eta})$.   
\end{enumerate}
\end{proposition}
\begin{proof} 
Noticing that 
$p|D^nv|^{p-2}v_{\alpha}\sigma^{ir}D_iv_{\alpha}
=\sigma^{ir}D_i|D^nv|^p$, by integration by parts and by Minkowski's 
and H\"older's inequalities we obtain that 
$$
\bar{\frak P}^{2}_{n,p}(t,v,g)
:=\sum_{r=1}^{\infty}(p|D^nv|^{p-2}D_\alpha v,\cM^rv_{\alpha}+g^r_{\alpha}))^2
$$
can be estimated by the right-hand side of \eqref{DM1estimate}. 
By Minkowski and H\"older's inequalities it is easy to see that 
$$
{\frak P}^{2}_{n,p}(t,v,g)-\bar{\frak P}^{2}_{n,p}(t,v,g)
$$ 
can also be 
estimated by the right-hand side of \eqref{DM1estimate}. 
To prove (ii) let 
$$                                                  
A_s(x,z):=|D^n(v+I^{\eta}v+h)|^{p}
-|D^nv|^p
-p|D^nv|^{p-2}D_{\alpha} v
D_{\alpha}(I^{\eta}v+h), 
$$
denote the integrand in \eqref{DM3}. 
Using Taylor's formula for 
$|{\bf x}+{\bf y}|^p-|{\bf x}|^p-p|\bf x|^{p-2}{\bf x}_{\alpha}{\bf y}_{\alpha}$ 
with vectors 
$$
{\bf x}_{\alpha}:=D_{\alpha}v, 
\quad 
{\bf y}_{\alpha}:=D_{\alpha}(I^{\eta}v+h), 
$$
$\alpha\in\{1,2,...,d\}^n$, we have the estimate 
$$
0\leq A_s(x,z)\leq N|{\bf x}|^{p-2}|{\bf y}|^2+N|{\bf y}|^p
\leq N'|D^nv|^{p-2}|D^nI^{\eta}v|^2
$$
$$
+N'|D^nv|^{p-2}|D^nh|^2+N'|D^nh|^p
+N'|D^n I^\eta v|^p
$$
with constants $N$ and $N'$ depending only on $d$,$p$ and $n$. 
By Fubini's theorem and H\"older's inequality
$$
\int_Z\int_{\bR^d}|D^nv|^{p-2}|D^nh|^2\,dx\,\mu(dz)
=\int_{\bR^d}|D^nv|^{p-2}|h(x)|^2_{\cL_2}\,dx
\leq |v|^{p-2}_{W^n_p}|h|^2_{W^n_p(\cL_{2})}.  
$$
By H\"older's inequality and Lemma \ref{lemma TIJ} we obtain 
$$
\int_{Z}\int_{\bR^d}
|D^nv|^{p-2}|D^nI^{\eta}v|^2\,dx\,\mu(dz)
\leq \int_{Z}|D^nv|^{p-2}_{L_p}
|D^nI^{\eta}v|^2_{L_p}\,\mu(dz)
\leq N^2K^2_{\eta} |v|^{p-2}_{W^{n}_p}|v|^2_{W^{n+1}_{p}}. 
$$
Moreover, by Assumption \ref{assumption eta} 
and Lemma \ref{lemma TIJ} we have
$$
\int_Z\int_{\bR^d}|D^n I^\eta v|^p \,dx\,\mu(dz)
\leq N^p K^{p-2}K_\eta^2|v|^p_{W^{n+1}_p}.
$$
Combining these inequalities and using Young's inequality we get 
\eqref{DM3estimate}. 
\end{proof}

\mysection{Proof of the main result}                                            \label{section mainproof}

\subsection{Uniqueness of the generalised solution}
Let Assumptions \ref{assumption L}, \ref{assumption xi},  
and \ref{assumption eta} hold with $m=0$. For a fixed 
$p\in[2,\infty)$ let $u^{(i)}=(u^{(i)}_t)_{t\in[0,T]}$ 
be $L_p$-valued generalised solutions to equation \eqref{eq1} 
with initial condition $u^{(i)}_0=\psi\in L_p$ for $i=1,2$. 
Then for $v=u^{(2)}-u^{(1)}$ by Lemma \ref{Ito Lp formula} we have that almost surely 
\begin{align}
y_t:=|v_t|_{L_p}^p = &
\int_0^tQ(s,v_s)
+\bar Q^{\xi}(s,v_s)+\bar Q^{\eta}(s,v_s)\,ds  \nonumber\\
&
+\int_0^t\int_Z\int_{\bR^d}P^{\eta}(s,z,v_{s-})(x)\,dx\,\pi(dz,ds)
+\zeta_1(t)+\zeta_2(t)                                                                \label{y}
\end{align}
for all $t\in[0,T]$, where $\zeta_1$ and $\zeta_2$ are local martingales 
defined by 
$$
\zeta_1(t):=p\int_0^t\int_{\bR^d}
|v_s|^{p-2} v_s
\cM^r_sv_s\,dx\,dw_s^r, 
$$
$$
\zeta_{2}(t):=p\int_0^t\int_Z\int_{\bR^d}
|v_{s-}|^{p-2}v_{s-}
I^{\eta}v_{s-}\,dx
\,\tilde{\pi}(dz,ds),     
$$
$Q(s,\cdot)$, $Q^{\eta}(s,\cdot)$ and $P^{\eta}(s,z,\cdot)$ 
are functionals on $W^1_p$, for each $(\omega,s)$ and $z$,   
defined by 
\begin{equation*}
Q(s,v):=
p\int_{\bR^d}-D_i(|v|^{p-2}v)a_s^{ij}D_jv
+\bar b^i_s|v|^{p-2}vD_iv+c_s|v|^{p}
+\tfrac{p-1}{2}|v|^{p-2}\sum_{r=1}^{\infty}
|\cM_s^rv|^2\,dx, 
\end{equation*}
\begin{equation}                                               \label{barQeta}
\bar Q^{\eta}(s,v)=p\int_{\bR^d}
-D_i(|v|^{p-2}v)\cJ_{\eta}^{i} v+|v|^{p-2}v\cJ_{\eta}^{0}v\,dx, 
\end{equation}
\begin{equation*}                                                      
P^{\eta}(s,z,v):=|v+I^{\eta}v|^{p}-|v|^p-p|v|^{p-2}vI^{\eta}v,   
\end{equation*}
and  $\bar Q^{\xi}(s,\cdot)$ is defined as 
$\bar Q^{\eta}(s,\cdot)$ in \eqref{barQeta}, but with $\xi$ 
in place of $\eta$. Recall that 
$\hat b^i=b-D_ja^{ij}$ and $\cJ_{\eta}^{i}$, 
$\cJ_{\eta}^{0}$ are defined by \eqref{def Jk}-\eqref{def J0}.

Note that due to the convexity of the function $|r|^p$, $r\in\bR$, we 
have 
\begin{equation}                                                         \label{positive}
P^{\eta}(s,z,v)(x)\geq0\quad\text{for all $(\omega,s,z,x)$}
\end{equation}
for real-valued functions $v=v(x)$, $x\in\bR^d$.  Together with 
the above functionals we need also to estimate 
the functionals $\frak Q(s,\cdot)$ and $\hat Q(s,\cdot)$ defined 
for each $(\omega,s)\in\Omega\times[0,T]$ by
$$
\frak Q(s,v):=\int_{Z}\int_{\bR^d}P^{\eta}(s,z,v)(x)\,dx\,\mu(dz), 
\quad
\hat Q^\eta(s,v):=\frak Q(s,v)+Q^{\eta}(s,v)
$$
for $v\in W^1_p$.  
\begin{proposition}                                                              \label{proposition uni}
Let Assumptions \ref{assumption L}, 
\ref{assumption xi} and \ref{assumption eta} hold with $m=0$. 
Then for $p\geq 2$ there are constants $N_1=N_1(d,p,K)$, 
$N_2=N_2(d,p,K,K_{\xi})$ and $N=N(d,p,K,K_{\eta})$ such that 
\begin{equation}                                                           \label{Qxieta}
Q(s,v)\leq N_1|u|^p_{L_p},
\quad
\bar Q^{\xi}(s,v)\leq N_2|v|^p_{L_p}, 
\quad 
\bar Q^{\eta}(s,v)\leq N|v|^p_{L_p},
\quad 
\hat Q^\eta(s,v)\leq N|v|^p_{L_p}, 
\end{equation}
\begin{equation}                                                            \label{P}
\frak Q(s,v)\leq N |v|^p_{W^1_p}
\end{equation}
for all $v\in W^1_p$ and $(\omega,s)\in\Omega\times[0,T]$.  
\end{proposition}
\begin{proof}
Notice that the estimate \eqref{P} is the special case 
of Proposition \ref{proposition DM} (iii), and 
for $v\in W^2_p$ the second and third estimates in \eqref{Qxieta} 
follow from the estimate \eqref{Jpestimate} 
in Proposition \ref{proposition Qxieta} (ii).  
Notice also that for $v\in W^2_p$ 
the first estimate in \eqref{Qxieta} is a special case of 
\eqref{estimate L} in 
Proposition \ref{proposition L}. If $v\in W^2_p$ then 
by Assumption \ref{assumption eta} and estimate \eqref{intJ} in Lemma \ref{lemma intIJ} we have
$$
\hat{Q}^\eta(s,v)=\int_Z\int_{\bR^d}p|v|^{p-2}v(J^\eta v-I^\eta v)+|v+I^\eta v|^p
-|v|^p\,dx\,\mu(dz)
$$
$$
=\int_Z\int_{\bR^d}J^\eta |v|^p\,dx\,\mu(dz)\leq CK^2_\eta|v|^p_{L_p}
$$
with a constant $C$ only depending on $K$ and $d$.
It is an easy exercise to show that 
the functionals in the left-hand side of the inequalities in \eqref{Qxieta} 
are continuous in $v\in W^1_p$, that completes the proof of the 
proposition. 
\end{proof} 
Define now the stochastic process
$$
\gamma_t:=|v_t|^p_{L_p}+\int_0^t|v_s|_{W^1_p}^p\,ds 
$$
and the stopping time 
$$
\tau_n:=\inf\{t\in[0,T]:\gamma_t\geq n\}\wedge\rho_{n}
$$
for every integer $n\geq1$, where $(\rho_n)_{n=1}^{\infty}$ is an increasing 
sequence of stopping times, converging to infinity such that 
$(\zeta_i(t\wedge\rho_n))_{t\in[0,T]}$ is a martingale for each $n\geq1$ and $i=1,2$. 
Then clearly, $E\zeta_i(t\wedge \tau_n)=0$ for $t\in[0,T]$ and $i=1,2$.  
Due to \eqref{positive} and the estimate in \eqref{P} we have 
$$
E\int_0^{T\wedge\tau_n}\int_{Z}\int_{\bR^d}
|P^{\eta}(s,z,v_{s-})(x)|\,dx\,\mu(dz)\,ds  
\leq NE\int_0^{T\wedge\tau_n}|v_s|^p_{W^1_p}\,ds<\infty, 
$$
which implies 
$$
E\int_0^{t\wedge\tau_n}\int_Z\int_{\bR^d}P^{\eta}(s,z,v_{s-})(x)\,dx\,\pi(dz,ds)
$$
$$
=E\int_0^{t\wedge\tau_n}\int_Z\int_{\bR^d}P^{\eta}(s,z,v_{s-})(x)\,dx\,\mu(dz)\,ds
=E\int_0^{t\wedge\tau_n}\frak Q(s,v_{s-})\,ds. 
$$
Thus, substituting $t\wedge\tau_n$ in place of $t$ in \eqref{y} and then 
taking expectation and using Proposition \ref{proposition uni} we obtain 
$$
Ey_{t\wedge\tau_n}=E\int_0^{t\wedge\tau_n}
Q(s,v_s)+Q^{\xi}(s,v_s)+\hat Q^{\eta}(s,v_s)
\,ds
$$
$$ 
\leq NE\int_0^{t\wedge\tau_n}|v_s|^p_{L_p}\,ds
\leq N\int_0^{t}Ey(s\wedge\tau_n)\,ds\leq NTn<\infty
$$
for $t\in[0,T]$. 
Hence by Gronwall's lemma $Ey(t\wedge\tau_n)=0$ for each $t\in[0,T]$ and 
integer $n\geq1$, which implies almost surely $y_t=0$ for all $t\in[0,T]$ 
and completes the proof of the uniqueness. 
\subsection{A priori estimates} 

\begin{proposition}                                                          \label{proposition apriori1}
Let Assumptions \ref{assumption L} through \ref{assumption eta} hold 
with an integer $m\geq0$. Assume $p=2^k$ for some integer $k\geq1$ and let 
$u=(u_t)_{t\in[0,T]}$ be a $W^{m+2}_p$-valued generalised solution to 
\eqref{eq1}-\eqref{ini1} such that it is cadlag as a $W^m_p$-valued 
process and 
$$
E\int_0^T|u_t|^p_{W^{m+2}_p}\,dt+E\sup_{t\leq T}|u_t|^p_{W^m_p}<\infty.
$$
Then 
\begin{equation}                                                        \label{apriori1}
E\sup_{t\leq T}|u_t|^p_{W^n_p}
\leq NE|\psi|^p_{W^n_p}+NE\cK^{p}_{n,p}(T)
\quad
\text{for every integer $n\in[0,m]$} 
\end{equation}                                                      
with a constant $N=N(m,d,p,T, K,K_{\xi},K_{\eta})$. 
\end{proposition}
\begin{proof}
We may assume that the right-hand side of the inequality 
\eqref{apriori1} is finite. 
  For multi-numbers $|\alpha|\leq m$ and $\varphi\in C_0^{\infty}$ 
\begin{align}                                              
d(D_{\alpha}u_t,\varphi)= &(D_{\alpha}\cA_tu_t(x)+D_{\alpha}f_t(x),\varphi)\,dt
+(D_{\alpha}\cM^r_tu_t(x)+D_{\alpha}g^r_t(x),\varphi)\,dw^r_t                        \nonumber\\
&+\int_Z(D_{\alpha}( u_{t-}(x+\eta_{t,z}(x))-u_{t-}(x)
+h_t(x,z)),\varphi)\,\tilde{\pi}(dz,dt)                                                                      \nonumber
\end{align}
Recall, see \eqref{L_p}, that by Lemma \ref{lemma Ito1} 
on It\^o's formula for each integer $n\in[0,m]$
\begin{equation*}                                                                                                      
d|D^nu_t|^p_{L_p}=(Q_{n,p}(v_t,t,f_t,g_t)+Q^{\xi}_{n,p}(v_t)+\hat Q_{n,p}(v_t,h_t))\,dt
+\sum_{i=1}^3d\zeta_i(t),
\end{equation*}
where the $Q_{n,p}$, $Q^{\xi}_{n,p}$ and $\hat Q_{n,p}$ are defined in 
\eqref{drift L}, \eqref{Qxi} and \eqref{drift}, and 
$\zeta_i=(\zeta_i(t))_{t\in[0,T]}$ is a cadlag local martingale starting from zero for each 
$i=1,2,3$, such that 
$$
d\zeta_1(t)=p(|D^nu_t|^{p-2}D_{\alpha}u_t,
D_{\alpha}\cM^r_tu_t+D_{\alpha}g^r_t)\,dw_t^r, 
$$
\begin{equation}                                               \label{z2}
d\zeta_{2}(t):=p\int_Z(
|D^nu_{t-}|^{p-2}D_{\alpha}u_{t-},
D_{\alpha}I^{\eta}u_{t-}+D_{\alpha}h_{t,z})\,\tilde\pi(dz,dt)
\end{equation}
and 
\begin{equation}                                               \label{z3}
d\zeta_{3}(t):=
\int_Z P_{n,p}(t,u_{t-},h_t)\,\pi(dz,dt)-\int_Z P_{n,p}(t,u_{t-},h_t)\,\mu(dz,dt), 
\end{equation}
where 
$$
P_{n,p}(t,v,h):=\int_{\bR^d}|D^n(T^{\eta}v+h)|^{p}-|D^nv|^p
-p|D^nv|^{p-2}
\sum_{|\alpha|=n}
D_{\alpha}vD_{\alpha}(I^{\eta}v+h) \,dx
$$
for $v\in W^{m+2}_p$ and $h\in W^{m+2}_{p,2}$. 
By Propositions \ref{proposition L} and \ref{proposition Qxieta}  
we obtain 
\begin{equation*}                                                        
d|D^nu_t|^p\leq N(|u_t|^{p}_{W^n_p}\,dt+d\cK^p_{n,p}(t))
+\sum_{i=1}^3d\zeta_i(t).
\end{equation*}
Hence using 
the estimate \eqref{DM3estimate} in Proposition \ref{proposition DM}
we have
$$
E|D^n u_{t\wedge \tau_k}|^p\leq E|D^n\psi|^p
+N\int_0^tE|u_{s\wedge \tau_k}|^p_{W^n_p}\,ds
+NE\cK^p_{n,p}(T\wedge \tau_k)
$$
for all $t\in[0,T]$, for a localising sequence $(\tau_k)_{k=1}^{\infty}$ 
of stopping times 
for $\zeta_i$, $i=1,2,3$. Hence  by Gronwall's lemma
$$
E|u_{t\wedge \tau_k}|^p_{W^n_p}\leq N(E|\psi|^p_{W^n_p}+E\cK^p_{n,p}(T))
$$
for $t\in[0,T]$ and $k\geq 1$ with a constant 
$N=N(d,m,p,T,K,K_\xi, K_\eta)$, which implies
\begin{equation}                                                                                        \label{sup0}
\sup_{t\leq T}E|u_t|^p_{W^n_p}\leq N(E|\psi|^p_{W^n_p}+E\cK^p_{n,p}(T))
\end{equation}
by Fatou's lemma. To show that we can interchange the supremum 
and expectation it suffices to prove 
that for every $\varepsilon>0$ 
\begin{equation}                                                             \label{sup1}
E\sup_{t\leq T}|\zeta_1(t)|\leq \varepsilon E\sup_{t\leq T}|u_t|^p_{W^n_p}
+
N(E|\psi|^p_{W^n_p}+E\cK^p_{n,p}(T))<\infty
\end{equation}
and 
\begin{equation}                                                              \label{sup23}
E\sup_{t\leq T}|\zeta_2(t)+\zeta_3(t)|\leq \varepsilon E\sup_{t\leq T}|u_t|^p_{W^n_p}
+
N(E|\psi|^p_{W^n_p}+E\cK^p_{n,p}(T))<\infty
\end{equation}
with a constant $N=N(\varepsilon,d,m,p,T,K,K_{\xi},K_{\eta})$. 
The proof of \eqref{sup1} is well-known and it goes as follows. 
By the Davis inequality, using the estimate in \eqref{DM1estimate}   
we obtain 
$$
E\sup_{t\leq T}|\zeta_1(t)|\leq 3E\left(\int_0^{T}
\frak P^2(t,u_t,g_t)\,dt\right)^{1/2}
\leq N
E\left(\int_0^{T}
|u_t|^{2p}_{W^n_p}+|u_t|^{2p-2}_{W^n_p}|g_t|^{2}_{W^{n}_p}\,dt\right)^{1/2}
$$
$$
\leq N
E\left(\sup_{t\leq T}|u_t|^{p}_{W^n_p}\int_0^{T}
|u_t|^{p}_{W^n_p}+|u_t|^{p-2}_{W^n_p}|g_t|^{2}_{W^{n}_p}\,dt\right)^{1/2}
$$
\begin{equation}                                                       \label{sup2}
\leq \varepsilon E\sup_{t\leq T}|u_t|^{p}_{W^n_p}
+\varepsilon^{-1} N^2E\int_0^T|u_t|^{p}_{W^n_p}+|g_t|^{p}_{W^{n}_p}\,dt<\infty,  
\end{equation}
which gives \eqref{sup1} by virtue of \eqref{sup0}. To prove \eqref{sup23} we first 
assume that $\mu$ is a finite measure. 
Then taking into account Lemma \ref{lemma B} we have 
$$
\zeta(t):=\zeta_2(t)+\zeta_3(t)
=\sum_{i=4}^8\zeta_i(t)
$$
with 
$$
\zeta_4(t)=\int_0^t\int_{Z}\int_{\bR^d}
|D^nu_s|^{p-2}D_{\alpha}u_sD_{\alpha}h_s\,dx\,\tilde\pi(dz,ds), 
$$
$$
\zeta_5(t)=\int_0^t\int_{Z}\int_{\bR^d}I^{\eta}|D^nu_s|^{p}\,dx\,\tilde\pi(dz,ds), 
$$
$$
\zeta_6(t)=p\int_0^t\int_{Z}\int_{\bR^d}T^{\eta}
(|D^nu_s|^{p-2}v_{\alpha})G^{(\alpha)}(u_s)
\,dx\,\tilde\pi(dz,ds)
$$
$$
\zeta_7(t)=p\int_0^t\int_{Z}\int_{\bR^d}
I^{\eta}(D^nu_{s-}|^{p-2}D_{\alpha}u_{s-})D_{\alpha}h_s\,dx\,\pi(dz,ds) 
$$
$$
-p\int_0^t\int_{Z}\int_{\bR^d}
I^{\eta}(D^nu_s|^{p-2}D_{\alpha}u_s)D_{\alpha}h_s\,dx\,\mu(dz)\,ds
$$
$$
\zeta_8(t)=\int_0^t\int_{Z}\int_{\bR^d}\bar B_n(u_{s-},h_s)\,dx\,\pi(dz,ds)-
\int_0^t\int_{Z}\int_{\bR^d}\bar B_n(u_{s-},h_s)\,dx\,\mu(dz)\,ds. 
$$
By Minkowski's and H\"older's inequalities  
$$
\int_0^T\int_Z\Big|\int_{\bR^d}
|D^nu_s|^{p-2}D_{\alpha}u_sD_{\alpha}h_s\,dx\Big|^2\mu(dz)\,ds
\leq \int_0^T\Big(\int_{\bR^d}
|D^nu_s|^{p-1}|D^nh_s|_{\cL_2}\,dx\Big)^2\,ds
$$
$$
\leq \int_0^T
|D^nu_s|^{2p-2}_{L_p}|h_s|^2_{W^n_{p,2}}\,ds. 
$$
Using this we can apply the Davis inequality to get 
\begin{equation}                                                            \label{sup4}
E\sup_{t\leq T}|\zeta_i(t)|\leq \varepsilon E\sup_{t\leq T}|u_t|^p_{W_p^n}
+\varepsilon^{-1}N(E|\psi|^p_{W^n_p}+E\cK^p_{n,p}(T))
\end{equation}
for $i=4$ in the same way as estimate in \eqref{sup1} is proved. 
Using Lemma  \ref{lemma intIJ} (iii) 
we get 
$$
\int_0^T\int_{Z}\Big|\int_{\bR^d}I^{\eta}|D^nu_s|^{p}\,dx\Big|^2\mu(dz)\,ds
\leq N\int_0^T|D^nu_s|^{2p}_{L_p}\,ds, 
$$
which allows us to get 
the estimate \eqref{sup4} for $i=5$.  
Using H\"older's inequality we get 
$$
\int_{\bR^d}|T^{\eta}(|D^nu_s|^{p-2}D_{\alpha}u_s)G^{(\alpha)}(u_s)|\,dx
\leq N\bar\eta|D^nu_s|^{p-1}_{L_p}|u_s|_{W^n_p}\leq N \bar\eta|u_s|^p_{W^n_p}. 
$$
Hence 
$$
\int_0^T\int_{Z}\Big|\int_{\bR^d}T^{\eta}(|D^nu_s|^{p-2}v_{\alpha})G^{(\alpha)}(u_s)
\,dx\Big|^2\,\mu(dz)\,ds
\leq N\int_0^T|u_s|^{2p}_{W^n_p}\,ds, 
$$
which gives  the estimate \eqref{sup4} for $i=6$.  
By Lemma \ref{lemma HK} and the estimate in \eqref{sup0} we have  
$$
E\int_0^T\int_{Z}\Big|\int_{\bR^d}
EI^{\eta}(D^nu_{s-}|^{p-2}D_{\alpha}u_{s-})D_{\alpha}h_s\,dx\Big|\mu(dz)\,ds= 
E\int_0^TH_{n,p}(u_{s-},h_s)\,ds
$$
$$
\leq NE\int_0^T|u_s|^p_{W^m_p}\,ds+NE\cK^p_{n,p}(T)
\leq N'E|\psi|^p_{W^n_p}+N'E\cK^p_{n,p}(T)
$$
with constants $N$ and $N'$ depending only on 
$K$, $d$, $m$, $p$, $T$, $K_{\xi}$ and $K_{\eta}$. Hence  
\begin{equation}                                                        \label{sup7}                                                                                           
E\sup_{t\leq T}|\zeta_i(t)|\leq N E|\psi|^p_{W^n_p}+NE\cK^p_{n,p}(T) 
\end{equation}
for $i=7$ with a constant $N=N(K,d,m,p,T,K_{\xi},K_{\eta})$. Similarly, using the estimate  
for $\bar B_{n,p}$ in \eqref{BB} and the estimate for $K_{n,p}$ 
in Lemma \ref{lemma HK} we obtain 
the estimate \eqref{sup7} for $i=8$. Clearly,  \eqref{sup4} for $i=4,5,6$ and 
\eqref{sup7} for $i=7,8$ imply estimate \eqref{sup2}. 

In the general case of $\sigma$-finite measure $\mu$ 
we have a nested sequence $(Z_n)_{n=1}^{\infty}$ 
of sets $Z_n\in\cZ$ such that $\mu(Z_n)<\infty$ 
for every $n$ and $U_{n=1}^{\infty}Z_n=Z$. 
For each integer $k\geq1$ define the measures 
$$
\pi_k(F)=\pi((Z_k\times(0,T])\cap F), \quad \mu_k(G)=\mu(Z_k\cap G)
$$
for $F\in \cZ\otimes\cB((0,T])$ and $G\in\cZ$, and set 
$\tilde\pi_k(dz,dt)=\pi_k(dz,dt)-\mu_k\otimes dt$. Let $\zeta^{(k)}_2$ 
and $\zeta^{(k)}_3$ be defined as $\zeta_2$ 
and $\zeta_3$, respectively, but with $\tilde\pi_k$, $\pi_k$ and $\mu_k$ 
in place of $\tilde\pi$, $\pi$ and $\mu$, respectively, in \eqref{z2} and 
\eqref{z3}. By virtue of what we have proved above, for each $k$ 
we have 
\begin{equation}                                                           \label{Zn}
E\sup_{t\leq T}|\zeta^{(k)}_2(t)+\zeta^{(k)}_3(t)|
\leq \varepsilon E\sup_{t\leq T}|u_t|^p_{W^n_p}
+
N(E|\psi|^p_{W^n_p}+E\cK^p_{n,p}(T))<\infty
\end{equation}
for $\varepsilon>0$ with a constant $N=N(\varepsilon,m,p,T,K,K_{\xi}, K_{\eta})$. 
Note that for a subsequence $k'\to\infty$ 
$$
\zeta^{(k')}_i(t)\to \zeta_i(t)\quad\text{almost surely, uniformly in $t\in[0,T]$}
$$
for $i=2,3$. Hence letting $k=k'\to\infty$ in \eqref{Zn} by Fatou's lemma we 
obtain \eqref{sup2}, which completes the proof of the lemma. 
\end{proof}

To obtain the estimate \eqref{apriori1} for an arbitrary $p\in[2,\infty)$ 
we make the following assumptions. 
\begin{assumption}                                              \label{assumption e1}
The initial condition $\psi$ and the free data $f$, $g$ and $h$ vanish 
if $|x|\geq R$ for some $R>0$. 
\end{assumption}
\begin{assumption}                                              \label{assumption e2}
Assumptions \ref{assumption L} through \ref{assumption free} hold 
for each integer $m\geq0$ with non-negative functions  
$\bar\xi=\bar\xi_m(z)$, $\bar\eta=\bar\eta_m(z)$ of $z\in Z$ 
and constants $K=K_m$,   
$$
K^2_{\xi}:=K^2_{\xi,m}=\int_Z\bar\xi^2_m(z)\,\nu(dz)<\infty, 
\quad
K^2_{\eta}:=K^2_{\eta,m}=\int_Z\bar\eta^2_m(z)\,\mu(dz)<\infty. 
$$
Moreover, 
$$
E|\psi|_{W^m_p}^p+E\cK^p_{m,p}(T)<\infty
\quad
\text{for each integer $m\geq0$}.
$$
\end{assumption}
\begin{assumption}                                              \label{assumption e3}
There is a constant $\varepsilon>0$ such that 
$P\otimes dt \otimes dx$-almost all
 $(\omega,t,x)\in\Omega\times H_T$
$$
(2a^{ij}-\sigma^{ir}\sigma^{jr})z^iz^j\geq \varepsilon|z|^2\quad
\text{for all $z=(z^1,...,z^d)$}.
$$
\end{assumption}
\begin{proposition}                                               \label{proposition apriori2}                        
Let Assumptions \ref{assumption e1}, \ref{assumption e2} and 
\ref{assumption e3} hold. Then \eqref{eq1}-\eqref{ini1} has a 
unique generalised solution $u=(u_t)_{t\in[0,T]}$. Moreover 
$u$ is a cadlag $W^n_p$-valued process for 
every integer $n\geq0$, and estimate \eqref{apriori1}
holds for each integer $m\geq1$ and real number $p\geq2$, 
with a constant $N=N(m,d,p,T, K_m,K_{\xi,m},K_{\eta,m})$. 
\end{proposition}
Before proving this proposition we introduce some 
notations.  
For integers $r>1$, real numbers $n\geq 0$ and $p\geq 2$ let $\bU^{n}_{r,p}$ 
denote  
the space of $H^n_p$-valued $\cF\otimes\cB([0,T])$-measurable  
functions $v$ on $\Omega\times [0,T]$ such that 
$$
|v|_{\bU^{n}_{r,p}}:=E\left(\int_0^T|v_t|^r_{H^n_p}\,dt\right)^{p/r}<\infty.
$$
The subspace of well-measurable functions 
$v:\Omega\times[0,T]\to H^n_p$ in $\bU^{n}_{r,p}$ is denoted by 
$\bV^{n}_{r,p}$. 
Set $\Psi^m_p:=L_p(\Omega, H^m_p)$, and recall from the Introduction 
the definition of  
the spaces $\bH^n_p(V)$ and $\bH^n_p=\bH^n_p(\bR)$ 
for separable Banach spaces $V$. 

\begin{proof}[Proof of Proposition \ref{proposition apriori2}] 
The uniqueness of the solution is proved above. 
Due to Assumptions \ref{assumption e1} and \ref{assumption e2}
$$
E|\psi|^2_{W^n_2}+E\cK^2_{n,2}(T)<\infty
$$
for each $n$. Hence 
by \cite{G1982} for $p=2$ 
the Cauchy problem \eqref{eq1}-\eqref{ini1} 
has a unique generalised solution $u$  
which is a $W^n_2$-valued cadlag process 
and for each integer $n\geq0$ there is a constant $N$ such that 
$$
E\sup_{t\leq T}|u_t|^2_{W^n_2}\leq N(E|\psi|^2_{W^{n}_2}+E\cK^2_{n,2}(T))<\infty
$$
Thus by Sobolev's embedding $u$ is a cadlag $W^n_p$-valued process 
for every $n$ such that 
$$
E\sup_{t\leq T}|u|^p_{W^n_p}<\infty. 
$$
In particular, $u$ is a generalised 
solution to \eqref{eq1}-\eqref{ini1} for the given $p$.  
Moreover, if $p=2^k$ for some integer $k\geq1$ then 
by estimate \eqref{apriori1} for $m\geq1$ and $n=0,1,...,m$ we have 
\begin{equation}                                                       \label{sup r} 
|u|_{\bU^n_{r,p}}\leq N(|\psi|_{\Psi^n_p}
+|f|_{\bH^n_p}+|g|_{\bH^{n+1}_p(l_2)}                                                                             
+|h|_{\bH^{n+i}_p(\cL_{p,2})} )                                                    
\end{equation}
for $i=1$ when $p=2$ and $i=2$ when $p>2$, 
with a constant $N=N(d,m,p,T,K_m,K_{\xi,m}, K_{\eta,m})$.   
This means the solution operator
$$
\bS:(\psi,f,g,h)\to u
$$
is a bounded operator from 
$\Psi^n_p\times \bH^n_p\times \bH^{n+1}_p(l_2)\times\bH^{n+i}_p(\cL_{p,2})$ 
into $\bU^n_{p,r}$ with operator norm 
smaller than a constant $N=N(d,m,p,T,K_m,K_{\xi,m}, K_{\eta,m})$ 
for integers  $n\in[0,m]$, $r>1$, for $i=1$ when $p=2$ and $i=2$ when $p>2$.  
If $p$ is not an integer power of 2 then we take an integer 
$k\geq1$ and a parameter 
$\theta\in(0,1)$ such that $p_0=2^k<p<2^{k+1}=p_1$ 
and $1/p=(1-\theta)/p_0+\theta/p_1$.  
By Theorem \ref{theorem i1}, 
\ref{theorem Lpq interpolation} and \ref{theorem H} we have
$$
\Psi^n_p=[\Psi^n_{p_0},\Psi^n_{p_1}]_{\theta}=L_p(\Omega, H^n_p), 
\quad
[\bH^n_{p_0},\bH^n_{p_1}]_{\theta}=\bH^n_p,
$$
$$
[\bH^{n}_{p_0}(l_2),\bH^{n}_{p_1}(l_2)]_\theta=\bH^{n}_p(l_2),
\quad
[\bH^{n}_{p_0}(\cL_{p_0,2}),\bH^{n}_{p_1}(\cL_{p_1,2})]_\theta
=\bH^{n}_p(\cL_{p,2}) ,
$$
and
$$
\mathbb{U}^n_{r,p}
=[\mathbb{U}^n_{r,p_0},\mathbb{U}^n_{r,p_1}]_{\theta}=\mathbb{U}^n_{r,p} 
$$
for any $n\geq0$ and $r>1$. Consequently, 
by Theorem \ref{theorem i1} (i)
$\mathbb S$ is continuous and \eqref{sup r} holds 
for the given $p$ for all $r>1$,  
where letting $r\to\infty$ 
gives 
$$
E\esssup_{t\in[0,T]} |u_t|^p_{H^n_p}\leq N(E|\psi|^p_{H^p}+E\cK_{n,p}^p(T))
\quad
\text{for $n=0,1,...,m$}
$$   
for each $m\geq1$ with a constant 
$N=N(d,m,p,T,K_m,K_{\xi,m}, K_{\eta,m})$. 
Since $u$ is a cadlag process with values in $H^n_p$, with almost surely no jump 
at $T$, we can 
change the essential supremum  to supremum here, which 
finishes the proof of the proposition. 
\end{proof}

\subsection{Existence of a generalised solution.}
In the whole section we assume that the conditions of 
Theorem \ref{theorem main} 
are in force. By standard stopping time argument we may assume that 
\begin{equation*}                                                           
E\cK^p_{p,m}(T)< \infty.
\end{equation*}
First we additionally assume that Assumption \ref{assumption e1} holds and that 
$m$ is an integer. 
Under these conditions 
we approximate the Cauchy problem \eqref{eq1}-\eqref{ini1}  
by mollifying all data and coefficients involved in it. 
For $\varepsilon\in(0,\varepsilon_0)$ 
 we consider the equation 
\begin{align}                                              
dv_t(x)= & \left(\cA_t^\varepsilon v_t(x)+f^{(\varepsilon)}_t(x)
\right)\,dt
+\left( \cM^{\varepsilon r}_t v_t(x)+g^{(\varepsilon) r}_t(x) \right)\,dw^r_t\nonumber\\
&+\int_Z\left( v_t(x+\eta^{(\varepsilon)}_{t,z}(x))-v_t(x)
+h^{(\varepsilon)}_t(x,z) \right)\,\tilde{\pi}(dz,dt),  \label{eq1 smooth} 
\end{align}
with initial condition 
\begin{equation}                                                         \label{epsilonini}
v_0(x)=\psi^{(\varepsilon)}, 
\end{equation}
where $\varepsilon_0$ is given in Lemma \ref{lemma diff2},
$$
\cM^{\varepsilon r}=\sigma^{(\varepsilon)ir}D_i+\beta^{(\varepsilon)r},
\quad
\cA^\varepsilon=\cL^\varepsilon
+\cN^{\xi^{(\varepsilon)}}+\cN^{\eta^{(\varepsilon)}}
$$
with operators 
$$
\cL^\varepsilon=a^{\varepsilon ij}D_{ij}+b^{(\varepsilon)i}D_i+c^{(\varepsilon)},
\quad a^\varepsilon=a^{(\varepsilon)}+\varepsilon\bI, 
$$
and $\cN^{\xi^{(\varepsilon)}}$ and $\cN^{\eta^{(\varepsilon)}}$ defined as 
$\cN^{\xi}$ and $\cN^{\eta}$ in \eqref{def N} with $\xi^{(\varepsilon)}$ and 
$\eta^{(\varepsilon)}$ in place of $\xi$ and $\eta$, 
respectively. 
Recall that  $v^{(\varepsilon)}$ denotes  
the mollification $v^{(\varepsilon)}=S^\varepsilon v$ of $v$ in $x\in\bR^d$ 
defined in \eqref{mollification}. Note that by virtue of standard properties of mollifications 
and by Lemmas \ref{lemma diff2} and \ref{lemma e} the conditions 
of Proposition \ref{proposition apriori2} are satisfied. Hence for $u^{\varepsilon}$, 
the solution of \eqref{eq1 smooth}-\eqref{epsilonini} we have 
\begin{equation}                                                     \label{ue}
|u^{\varepsilon}|_{\bV^n_{r,p}}
\leq 
N(|\psi|_{\Psi^n_p}+|f|_{\bH^n_p}+|g|_{\bH^{n+1}_p(l_2)}+|h|_{\bH^{n+i}_p(\cL_{p,2})})
\quad
\text{for $n=0,1,2,...,m$}
\end{equation}
for every integer $r>1$ 
with a constant $N=N(d,p,m,T,K,K_{\xi},K_{\eta})$, where $i=1$ when $p=2$ 
and $i=2$ when $p>2$.  
Since $\bV^n_{r,p}$ is reflexive, there exists a sequence 
$\{\varepsilon_k\}_{k=1}^\infty$ 
and a process $u\in\bV^{n,r}_p$ such that 
$\lim_{k\to\infty}\varepsilon_k=0$ 
and $u^{\varepsilon_k}$ converges weakly to some $u$ in $\bV^{n,r}_p$.  
To show that a modification of $u$ is a solution to \eqref{eq1}-\eqref{ini1} we pass to the limit 
in the equation 
\begin{align}                                              
(u^{\varepsilon}_t,\varphi)= &(\psi^{(\varepsilon)},\varphi)+\int_0^t 
\langle\cA_s^\varepsilon u_s^{\varepsilon}+f^{(\varepsilon)}_s,\varphi\rangle
\,ds
+ \int_0^t
(\cM^{\varepsilon r}_su_s^{\varepsilon}+g^{(\varepsilon) r}_s,\varphi)
\,dw^r_s     \nonumber\\
&
+\int_0^t\int_Z\int_{\bR^d}
\big(u^{\varepsilon}_s(x+\eta^{(\varepsilon)}_{s,z}(x))-u^{\varepsilon}_s(x)
+h^{(\varepsilon)}_s(z)\big)\varphi(x)\,dx\,\tilde{\pi}(dz,ds)                               \label{eq1e} 
\end{align}
where $\varphi\in C_0^{\infty}$. 
To this end we take a bounded predictable 
real-valued process $\zeta=(\zeta_t)_{t\in[0,T]}$, multiply  both sides 
of equation \eqref{eq1e} with $\zeta_t$ 
and then integrate the expression we get against $P\otimes dt$ over $\Omega\times [0,T]$. 
Thus we obtain 
\begin{align}
F(u^{\varepsilon})= & E\int_0^T\zeta_t(\psi^{(\varepsilon)},\varphi)\,dt
+\sum_{i=1}^3F^i_{\varepsilon}(u^{\varepsilon})
+E\int_0^T\int_0^t \zeta_t(f_s^{(\varepsilon)},\varphi)
\,ds\,dt\nonumber\\
&
+E\int_0^T\zeta_t\int_0^t(g_s^{(\varepsilon)r},\varphi)\,dw_s^r\,dt
+E\int_0^T\zeta_t\int_0^t\int_Z(h_s^{(\varepsilon)},\varphi)\,\tilde{\pi}(dz,ds)\,dt,      \label{F1}
\end{align}
where $F$ and $F^i_{\varepsilon}$, $i=1,2,3$, 
are linear functionals of $v\in \bH^1_p$, defined by 
$$
F(v)=E\int_0^T\zeta_t(v_t,\varphi)\,dt, 
\quad 
F_{\varepsilon}^1=E\int_0^T\zeta_t\int_0^t\<\cA_s^{\varepsilon},\varphi\>
\,ds\,dt
$$
$$
F^2_{\varepsilon}(v)
=E\int_0^T\zeta_t\int_0^t(\cM_s^{\varepsilon r}v_s,\varphi)\,dw^r_s\,dt 
$$
and
$$
F^3_{\varepsilon}(v)
=
E\int_0^T\zeta_t\int_0^t\int_Z
(I^{\eta^{(\varepsilon)}}v_s,\varphi)\,\tilde{\pi}(dz,ds)\,dt.
$$
For each $i$, we also define the functional $F^i$ 
in the same way as $F^i_{\varepsilon}$ is defined above, 
but with $\cA$, $\cM$ and $I^{\eta}$ in place of 
$\cA^{\varepsilon}$, $\cM^{\varepsilon}$ 
and $I^{\eta^{(\varepsilon)}}$ respectively. Obviously, 
by H\"older's inequality and the boundedness of $\zeta$, 
for all $v\in\bV^1_p$ we have 
$$
|F(v)|\leq C|v|_{\bL_p}|\varphi|_{L_q}\leq C|v|_{\bH^1_p}|\varphi|_{L_q}
$$
with $q=p/(p-1)$ and a constant $C$ independent of $v$ and $k$, 
which means $F\in {\bH^{1}_p}^*$, 
the space of all bounded linear functionals on $\bH^1_p$. 
Next we show that 
$F^i_{\varepsilon}$ and $F^i$ are also in ${\bH^1_p}^*$,  
and $F^i_{\varepsilon}\to F^i$ strongly in ${\bH^1_p}^*$ 
as $\varepsilon\to0$ for $i=1,2,3$.
\begin{lemma}
For $i=1,2,3$ the functionals $F^i$ and $F^i_{\varepsilon}$ 
are in $\bH^{1*}_p$ for sufficiently small $\varepsilon>0$.
\end{lemma}
\begin{proof}
It is easy to show, see  Lemma 5.3 in \cite{DGW}, that  we have 
$F^1_{\varepsilon}\in {H^1_p}^*$.
Then due to the boundedness of $\zeta$, $\sigma^{(\varepsilon)r}$ 
and $\beta^{(\varepsilon)r}$, by the Davis and H\"older's inequalities, we get
$$
|F^2_{\varepsilon}(v)|
\leq 
CE\Big(\int_0^T\sum_{r}|(\cM_s^{\varepsilon r},\varphi)|^2\,ds\Big)^{1/2}
\leq C'|v|_{\bH^1_p}|\varphi|_{L_q}
$$
with constants $C$ and $C'$ independent of $v$ and $\varepsilon$. 
Similarly, by the boundedness of $\zeta$, using Lemma \ref{lemma TIJ}, 
and Davis' and H\"older's inequalities
$$
|F^3_{\varepsilon}(v)|
\leq 
CE(\int_0^T\int_Z|(I^{\eta^{(\varepsilon)}}v_s,\varphi)|^2\,\mu(dz)\,ds)^{1/2}
$$
$$
\leq 
CE\Big(
\int_0^T\int_Z\bar{\eta}^2(z)|v_s|^2_{H^1_p}|\varphi|^2_{L_q}\,\mu(dz)\,ds
\Big)^{1/2}
\leq C'|v|_{\bH^1_p}|\varphi|_{L_q}
$$
with constants $C$ and $C'$ independent of $v$ and $\varepsilon$. 
In the same way we can prove $F^i\in {\bH^1_p}^*$ for $i=1,2,3$.
\end{proof}

\begin{lemma}
For each $i=1,2,3$
\begin{equation}                                                            \label{Fi}
\lim_{\varepsilon\to0}\sup_{|v|_{\bH^1_p}\leq 1}
|(F^i_{\varepsilon}-F^i)(v)|=0.
\end{equation}
\end{lemma}

\begin{proof}
It is easy to show, see the proof of Lemma 5.4 in \cite{DGW}, that 
$$
\lim_{\varepsilon\to \infty}\sup_{|v|_{\bH^1_p}\leq 1}|(F^1_{\varepsilon}-F^1)(v)|=0.
$$
By the boundedness of $\zeta$ and using Davis' and H\"older's inequalities 
we have 
\begin{align*}
|F^2_{\varepsilon}(v)-F^2(v)| 
& 
\leq CE
\Big(
\int_0^T|(\cM^{\varepsilon r}_s v_s-\cM^r_s v_s,\varphi)\,ds)|^2
\Big)^{1/2}                                                                                                        \\
&
\leq CE
\Big(
\int_0^T\sum_{r=1}^\infty(|\sigma_s^{(\varepsilon)r}-\sigma_s^r|
|Dv_s|,|\varphi|)^2\,ds
\Big)^{1/2}                                                                                                          \\
&
\quad +CE\Big(\int_0^T\sum_{r=1}^\infty(|g_s^{(\varepsilon)r}-g_s^r|
|v_s|,|\varphi|)^2\,ds\Big)^{1/2}\\
&
\leq C(A^1_k(v)+A^2_k(v))
\end{align*}
for $v\in \bH^1_p$ and all integers $k\geq 1$ with a constant $C$ 
independent of $v$ and $\varepsilon$, where
$$
A^1_k(v):=E(\int_0^T|Dv_s|^2_{L_p}||\sigma_s^{(\varepsilon_k)}
-\sigma_s||\varphi||^2_{L_q})^{1/2}
$$
and
$$
A^2_k(v)
:=E(\int_0^T|v_s|^2_{L_p}||g^{(\varepsilon_k)}_s-g_s||\varphi||_{L_q}^2\,ds)^{1/2} 
$$
with $q=p/(p-1)$. By standard properties of mollification 
$$
|\sigma_t^{(\varepsilon)}-\sigma_t|+|g_t^{(\varepsilon)}-g_t|\leq N\varepsilon
$$
for all $\varepsilon\in(0,1)$ and $(x,t,\omega)\in H_T\times \Omega$ 
with a constant $N=N(K,d)$. Thus,
$$
\sup_{|v|_{\bH^1_p}\leq 1}A^1_k(v)+\sup_{|v|_{\bH^1_p}\leq 1}A^2_k(v)
\leq NT^{(p-2)/2p}\varepsilon|\varphi|_{L_q}
$$ 
with $q=p/(p-1)$ and a constant $N=N(K,d)$. Consequently, letting 
$\varepsilon\to0$ we obtain \eqref{Fi} for $i=2$. 
By the boundedness of $\zeta$,  
using Davis' inequality we get
\begin{equation}                                               \label{F3 1}
|F^3_{\varepsilon}(v)-F^3(v)|
\leq C E
\Big(
\int_0^T\int_Z|(I^{\eta^{(\varepsilon)}}v_s
-I^\eta v_s,\varphi)|^2\,\mu(dz)\,ds
\Big)^{1/2}
\end{equation}
for $v\in \bH^1_p$ with a constant $C$ independent of $\varepsilon$ and $v$. 
By Taylor's formula 
\begin{equation*}                                             
(I^{\eta^{(\varepsilon)}}v_s-I^\eta v_s,\varphi)
=\int_{\bR^d}\int_0^1v_i(\gamma^{\varepsilon}_{\theta}(s,z,x))
(\eta^{(\varepsilon)i}_{s,z}(x)-\eta^i_{s,z}(x))\varphi(x)\,d\theta\,dx
\end{equation*}
where $v_i:=D_iv$ and 
$$
\gamma^{\varepsilon}_{\theta}(s,z,x)
:=x+\theta\eta^{(\varepsilon)}_{s,z}(x)+(1-\theta)\eta_{s,z}(x)
$$
for all $\theta\in(0,1)$, $\varepsilon$ and 
$(s,z,\omega)\in [0,T]\times Z \times \Omega$. 
Then by Lemma \ref{lemma diff2} 
there are positive constants $\varepsilon_0$ 
and $M=M(K,d,m)$ such that for 
$\varepsilon_k\in(0,\varepsilon_0)$ and $\theta\in(0,1)$ 
the function $\gamma^{\varepsilon}_{\theta}(s,z,\cdot)$  
is a $C^{\bar m}$-diffeomorphism on $\bR^d$ and 
\begin{equation*}                                             \label{F3 3}
|D\gamma^{\varepsilon}_{\theta}(s,z,x)|\leq M\quad \text{for $x\in\bR^d$} 
\end{equation*}                                                           
for $(s,z,\omega)\in [0,T]\times Z \times \Omega$.  
Due to Assumption \ref{assumption eta} we have 
\begin{equation*}                                                   \label{F3 4}
|\eta^{(\varepsilon)}_{s,z}(x)-\eta_{s,z}(x)|
\leq \varepsilon \bar{\eta}(z)
\quad
\text{for all $\varepsilon>0$ and $(s,z,\omega,z)\in H_T\times \Omega \times Z$}. 
\end{equation*}
Thus from \eqref{F3 1} using H\"older's inequality we get
$$
|F^3_{\varepsilon}(v)-F^3(v)|
\leq CE
\Big(\int_0^T\int_Z|Dv(\gamma^{\varepsilon}_{\theta}(s,z))|^2_{L_p}|\varphi|^2_{L_q}
\varepsilon^2 \bar{\eta}^2(z)\,\mu(dz)\,ds
\Big)^{1/2}
$$ 
$$
\leq C'\varepsilon|\varphi|_{L_q}|v|_{\bH^1_p}|\bar{\eta}|_{\cL_2}
$$
with a constant $C'$ 
independent of $\varepsilon$ and $v$, which implies
$$
\lim_{\varepsilon\to0}\sup_{|v|_{\bH^1_p}\leq 1}
|F^3_{\varepsilon}(v)-F^3(v)|=0.
$$
\end{proof}
Since $F^i_{\varepsilon}\to F^i$ strongly in ${\bH^1_p}^*$ 
as $\varepsilon\to0$ and $u^{\varepsilon_k}$ to $u$ in $\bH^1_p$
for $\varepsilon_k\to0$, we have 
$$
\lim_{k\to \infty}F(u^{\varepsilon_k})=F(u),\quad 
\lim_{k\to\infty} F^i_k(u^{\varepsilon_k})=F^i(u) \quad \text{for} \quad i=1,2,3.
$$
By well-known properties of mollifications and using 
Lemma \ref{lemma e} it is easy to show
$$
\lim_{k\to\infty}E\int_0^T\zeta_t(\psi^{(\varepsilon_k)},\varphi)\,dt
=E\int_0^T\zeta_t(\psi,\varphi)\,dt,
$$
$$
\lim_{k\to\infty} E\int_0^T\int_0^t \zeta_t(f_s^{(\varepsilon_k)}
,\varphi)\,ds\,dt
=E\int_0^T\int_0^t \zeta_t(f_s,\varphi)
\,ds\,dt,
$$
$$
\lim_{k\to\infty}E\int_0^T\zeta_t\int_0^t(g_s^{(\varepsilon_k)r},\varphi)\,dw_s^r\,dt
=E\int_0^T\zeta_t\int_0^t(g_s^r,\varphi)\,dw_s^r\,dt,
$$
and 
$$
\lim_{k\to\infty}E\int_0^T\zeta_t\int_0^t\int_Z(h_s^{(\varepsilon_k)},\varphi)
\,\tilde{\pi}(dz,ds)\,dt  
=E\int_0^T\zeta_t\int_0^t\int_Z(h_s,\varphi)\,\tilde{\pi}(dz,ds)\,dt.  
$$
Hence, taking $k\to\infty$ in equation \eqref{F1} we get 
\begin{align*}
E\int_0^T\zeta_t(u_t,\varphi)\,dt = & E\int_0^T\zeta_t(\psi,\varphi)\,dt
+E\int_0^T\zeta_t\int_0^t\<\cA u_s,\varphi\>\,ds\,dt\\
&
+E\int_0^T\zeta_t\int_0^t(f_s,\varphi)\,ds\,dt
+E\int_0^T\zeta_t\int_0^t(\cM^r_su_s+g^r_s,\varphi)\,ds\,dt\\
&
+E\int_0^T\zeta_t\int_0^t\int_Z(I^\eta u_s+h_s,\varphi)\,\tilde{\pi}(dz,ds)\,dt
\end{align*}
for every bounded predictable process $\zeta$ 
and every $\varphi\in C^\infty_0$, 
which implies that for every $\varphi\in C^\infty_0$ 
equation \eqref{eq1} holds $P\otimes dt$ almost everywhere. 
Hence, by Lemma \ref{Ito Lp formula} $u$ has an $L_p$-valued cadlag 
modification, denoted also by $u$, which is a generalised solution to 
\eqref{eq1}-\eqref{ini1}. Moreover, from \eqref{ue} we obtain 
\begin{equation*}
|u|_{\bV^n_{r,p}}
\leq 
\liminf_{\varepsilon_k\to0}|u^{\varepsilon_k}|_{\bV^n_{r,p}}
\leq 
N(|\psi|_{\Psi^n_p}+|f|_{\bH^n_p}
+|g|_{\bH^{n+1}_p(l_2)}+|h|_{\bH^{n+i}_p(\cL_{p,2})})
\end{equation*}
for $n=0,1,...,m$ for every integer $r>1$ 
with a constant $N=N(d,p,m,T,K,K_{\xi},K_{\eta})$, 
where $i=1$  when $p=2$ and $i=2$ for $p>2$. 
Letting here $r\to\infty$ we obtain 
\begin{equation}                                        \label{spestimate}
E\esssup_{t\in[0,T]}|u_t|_{H^s_{p}}^p
\leq N(E|\psi|^p_{H^s_p}+E\cK^p_{s,p}(T))
\end{equation}
for $s=0,1,2,...m$ with a constant $N=N(d,p,m,T,K,K_{\xi},K_{\eta})$. 
We already know that $u$ is and $L_p$-valued cadlag process. 
Hence, applying Lemma \ref{lemma w} with $V=H^m_p$, $U=H^0_p$    
and we obtain that $u$ is weakly cadlag 
as an $H^m_p$-valued process, and we can change the essential 
supremum into supremum in \eqref{spestimate}, i.e.,
\begin{equation}                                                                                    \label{sup19}                                    
E\sup_{t\in[0,T]}|u_t|_{H^s_{p}}^p
\leq N(E|\psi|^p_{H^s_p}+E\cK^p_{s,p}(T))
\end{equation}
for $s=0,1,2,...,m$.  

To dispense with Assumption \ref{assumption e1}  
we take a non-negative function $\chi\in C^\infty_0(\bR^d)$ 
such that $\chi(x)=1$ for $|x|\leq 1$ 
and $\chi(x)=0$ for $|x|\geq 2$, and for integers $n\geq1$ define 
$$
\psi^n(x)=\psi(x)\chi_n(x), \quad f^n_t(x,z)=f_t(x,z)\chi_n(x),
$$
$$
g^{nr}_t(x)=g^r_t(x)\chi_n(x), \quad h^n_t(x,z)=h_t(x,z)\chi_n(x)
$$
for all $t\in[0,T]$, $x\in \bR^d$, $z\in Z$, where $\chi_n(x)=\chi(x/n)$. 
Then for each $n$ there is a unique generalised solution 
$u^n=(u^n_t)_{t\in[0,T]}$ to equation \eqref{eq1}-\eqref{ini1} 
with $\psi^n$, $f^n$, $g^n$ and $h^n$ in place of 
$\psi$, $f$, $g$ and $h$,  respectively. Moreover by \eqref{sup19} 
$$
E\sup_{t\leq T}|u^n(t)-u^l(t)|^p_{H^{m}_p}\leq NE|\psi^n-\psi^l|^p_{H^m_p}
$$
\begin{equation*}                                                 
+E\int_0^T
|f^n_s-f^l_s|^p_{H^m_p}
+|g^n_s-g^l|^p_{H^{m+1}_p(l_2)}+|h^n_s-h^l_s|^p_{H^{m+1}_p(\cL_{p,2})}
+{\bf 1}_{p>2}|h^n_s-h^l_s|^p_{H^{m+2}_p(\cL_{p,2})}\,ds
\end{equation*}
with a constant $N=N(T,p,m,K,K_\eta)$. Letting here $l,n\to\infty$ 
we get 
$$
\lim_{n,l\to\infty}E\sup_{t\leq T}|u^n(t)-u^{l}(t)|^p_{H^m_p}=0  
$$
Consequently,  $u^n(t)$ strongly converges to some $u=(u(t))_{t\in[0,T]}$ in $W^m_p$, 
uniformly in $t\in[0,T]$. Hence $u$ is an $L_p$-valued cadlag process,
and it is easy to show that it is a generalised solution to \eqref{eq1}-\eqref{ini1}
such that \eqref{sup19} holds, which implies    
\begin{equation*}                                        \label{spestimate19}
|u_t|^p_{\bU^n_{p,r}}
\leq N(|\psi|^p_{\Psi^n_p}+E\cK^p_{n,p}(T))\quad\text{for $n=0,1,2,...m$} 
\end{equation*}
for integers $r>1$. 
If $m\geq1$ is not an integer, then  we set $\theta=m-\lfloor m\rfloor$  
and by Theorems \ref{theorem i1}, \ref{theorem Lpq interpolation} 
and \ref{theorem H} we have 
$$
\Psi^m_p=[\Psi^{\lfloor m \rfloor}_p,\Psi^{\lceil m \rceil}_p]_{\theta}=L_p(\Omega, H^m_p), 
\quad
[\bH^{\lfloor m \rfloor}_{p},\bH^{\lceil m \rceil}_{p}]_{\theta}=\bH^m_p,
$$
$$
[\bH^{\lfloor m \rfloor+1}_{p}(l_2),\bH^{\lceil m \rceil+1}_{p}(l_2)]_\theta
=\bH^{m+1}_p(l_2),
\quad
[\bH^{\lfloor m \rfloor+i}(\cL_{p,2}),\bH^{\lceil m \rceil+i}(\cL_{p,2})]_\theta
=\bH^{m+i}_{p,2}(\cL_{p,2}) 
$$
for $i=1,2$, and
$$
\mathbb{U}^{m}_{r,p}
=[\mathbb{U}^{\lfloor m \rfloor}_{r,p},\mathbb{U}^{\lceil m \rceil}_{r,p}]_{\theta}
$$
for integers $r>1$. If Assumptions \ref{assumption L},  
through \ref{assumption free} with $m\geq1$ hold then, we have shown above 
that the solution 
operator $\bS$, which maps the data $(\psi,f,g,h)$ 
into the generalised solution $u$ of \eqref{eq1}-\eqref{ini1}  is continuous from 
$$
\Psi^{\lfloor m \rfloor}_p\times\bH^{\lfloor m \rfloor}_p
\times \bH^{\lfloor m \rfloor+1}_p(l_2)
\times \bH^{\lfloor m \rfloor+i}_{p,2}
$$
to $\mathbb{U}^{\lfloor m\rfloor}_{p,r}$, 
and from 
$$
\Psi^{\lceil m \rceil}_p\times\bH^{\lceil m \rceil}_p
\times \bH^{\lceil m \rceil+1}_p(l_2)
\times \bH^{\lceil m \rceil+i}_{p,2}
$$
to $\mathbb {U}^{\lceil m\rceil}_{p,r}$, 
for $i=1$ when $p=2$ and for $i=2$ when $p>2$, 
with operator norms bounded by a constant 
$N=N(d,p,m,T,K,K_{\xi},K_{\eta})$. 
Hence by Theorem \ref{theorem i1} (i) we have  
\begin{equation*}                                      
|u|^p_{\bU^{m}_{r,p}}
\leq N(E|\psi|^p_{H^m_p}+E\cK^p_{m,p}(T))
\end{equation*}
with a constant $N=(p,d,m,T,K,K_{\eta})$. 
In the same way we get 
\begin{equation*}                                        \label{estimate_ps}
|u|^p_{\bU^{s}_{r,p}}
\leq N(E|\psi|^p_{H^s_p}+E\cK^p_{s,p}(T)) \quad\text{for any $s\in[0,m]$}.  
\end{equation*}
Now, like before, letting here $r\to\infty$ we obtain \eqref{spestimate} 
for real numbers $s\in[0,m]$, and using Lemma \ref{lemma w} we get 
that $u$ is an $H^m_p$-valued weakly cadlag process such that 
\eqref{sup19} holds for any $s\in[0,m]$.    
Taking into account that $u$ is a strongly cadlag $L_p$-valued process 
and using the interpolation inequality 
Theorem \ref{theorem i1}(v) with $A_0:=L_p$ and $A_1:=H^m_p$,  
we get that $u$ is strongly cadlag as an $H^s_p$-valued process 
for every real number $s<m$. 

Finally we can prove estimate \eqref{estimate main} 
for $q\in(0, p)$ 
by applying Lemma \ref{lemma sup} in the same way 
as it is used in \cite{GK2003} to prove the corresponding 
supremum estimate.

\smallskip
\noindent 
 {\bf Acknowledgement.} The main result of this paper 
was presented at the conference on ``Harmonic Analysis for Stochastic PDEs" in Delft, 
10-13 July, 2018,  and at  the ``9th International Conference 
on Stochastic Analysis and Its Applications" 
in Bielefeld, 3-7 September, 2018. 
The authors are grateful to the organisers of these meetings, and to Konstantinos Dareiotis, 
Stefan Geiss and Mark Veerar for useful discussions and precious remarks.

\end{document}